\documentclass{amsart}
\usepackage{amsmath}
\usepackage{amssymb}
\usepackage{graphicx}
\newtheorem{remark}{Remark}[section]
\newtheorem{theorem}{Theorem}[section]
\newtheorem{lemma}{Lemma}[section]
\newtheorem{definition}{Definition}[section]
\newtheorem{algorithm}{Algorithm}[section]
\numberwithin{equation}{section}

\newcommand{\bv}{{\bf v}}

\def\T{{\mathcal T}}
\def\E{{\mathcal E}}

\def\vn{{\bf n}}
\def\vt{{\bf t}}

\def\3bar{{|\hspace{-.02in}|\hspace{-.02in}|}}

\newcommand{\Forall}{\textrm{for all }}
\newcommand{\bbK}{\mathbb K}
\newcommand{\bbR}{\mathbb R}
\newcommand{\bbT}{\mathbb T}
\newcommand{\bbQ}{\mathbb Q}

\newcommand{\hvn}{\hat{\vn}}
\newcommand{\hvt}{\hat{\vt}}

\newcommand{\eu}{{e_{\vu}}}
\newcommand{\ep}{{e_p}}
\newcommand{\heu}{{\hat{e}_{\vu}}}
\newcommand{\hep}{{\hat{e}_p}}

\newcommand{\circES}{\E_{0,h}^S}
\newcommand{\circED}{\E_{0,h}^D}

\renewcommand{\div}{\textrm{div}}

\newcommand{\vf}{\boldsymbol{f}}

\newcommand{\vu}{\boldsymbol{u}}
\newcommand{\vv}{\boldsymbol{v}}
\newcommand{\vw}{\boldsymbol{w}}

\newcommand{\vV}{\boldsymbol{V}}

\newcommand{\vzero}{\boldsymbol{0}}

\begin{document}
\title{Weak Galerkin method for the coupled Darcy-Stokes flow}
\author{Wenbin Chen}%
\address{Department of Mathematics, Fudan University, Shanghai, China}%
\email{wbchen@fudan.edu.cn}%
\author{Fang Wang}
\address{Department of Mathematics, Fudan University, Shanghai, China}%
\email{07300180148@fudan.edu.cn}%
\author{Yanqiu Wang} %
\address{Department of Mathematics, Oklahoma State University, Stillwater, OK, USA} %
\email{yqwang@math.okstate.edu} %

% ----------------------------------------------------------------
\begin{abstract}
A family of weak Galerkin finite element discretization is developed
for solving the coupled Darcy-Stokes equation. The equation in
consideration admits the Beaver-Joseph-Saffman condition on the
interface. By using the weak Galerkin approach, in the discrete
space we are able to impose the normal continuity of velocity
explicitly. Or in other words, strong coupling is achieved in the
discrete space. Different choices of weak Galerkin finite element
spaces are discussed, and error estimates are given.
\end{abstract}
\maketitle

%%%%%%%%%%%%%%%%%%%%%%%%%%%%%%%%%%%%%%%%%%%%%%%%%%%%%%%%%%%%%%%%%%%%%
%%%%%%%%%%%%%%%%%%%%%%%%%%%%%%%%%%%%%%%%%%%%%%%%%%%%%%%%%%%%%%%%%%%%%
%%%%%%%%%%%%%%%%%%%%%%%%%%%%%%%%%%%%%%%%%%%%%%%%%%%%%%%%%%%%%%%%%%%%%
\section{Introduction}
The goal of this paper is to propose and analyze a weak Galerkin
finite element discretization for the coupled Darcy-Stokes equation.
The coupled Darcy-Stokes problem has many applications. Readers may
refer to the nice overview \cite{DiscacciatiQuarteroni09} and
references therein for its physical background, modeling, and common
numerical methods. To solve the coupled Darcy-Stokes equation
numerically, one must address two important issues: how to
approximate the Darcy-Stokes interface conditions and how to couple
the discretization on both the Darcy side and the Stokes side. Below
we shall briefly state how these two issues will be addressed in the
proposed weak Galerkin method.

In this paper, we consider the Beavers-Joseph-Saffman (BJS)
interface condition \cite{Saffman,JagerMikelic96,JagerMikelic00,JagerMikelic01,PayneStraughan},
which is easier to handle than the original Beavers-Joseph interface condition \cite{BeaversJoseph},
as the BJS condition will generate a coercive bilinear form in the variational formulation.
The BJS interface condition works well when the flow in the porous region is small comparing to the Stokes flow,
around the interface.

In the Darcy region, one can use either the primal formulation, which only involves the pressure,
or the mixed formulation, which involves both the flux and the pressure, to model the problem.
Here we choose the mixed formulation, which has been studied in
\cite{Arbogast07,Bernardi08,Chen14,GalvisSarkis07,Gatica09,Gatica11,KanschatRiviere,Karper08,Layton03,Marquez14,Riviere05,RiviereYotov05}.
In \cite{Layton03}, rigorous analysis of the mixed formulation and
its weak existence have been presented. 
According to whether the normal continuity of the velocity is explicitly enforced on the interface or not,
two different formulations are proposed in \cite{Layton03}: 
a strongly coupled formulation and a weakly coupled formulation.
Various numerical discretizations have been developed for these two mixed formulations:
the work in \cite{Gatica09,Gatica11,Layton03} are based on the
weakly coupled formulation, while the work in
\cite{Arbogast07,Chen14,Karper08,KanschatRiviere,Riviere05,RiviereYotov05}
are based on the strongly coupled formulation. 
We will adopt the strongly coupled formulation in this paper, which imposes the normal continuity 
of the velocity strongly in the functional space.
Finally, it is also worth
mentioning that a different treatment of the interface condition in the variational formulation is to
use the idea of mortar elements, which gives a ``strong'' but not
pointwise coupling \cite{Bernardi08,GalvisSarkis07,Marquez14}.

For the discretization of both the Darcy side and the Stokes side,
we use the weak Galerkin finite element.
The weak Galerkin method was recently introduced in
\cite{WangYe_PrepSINUM_2011} for second order elliptic equations. It
is an extension of the standard Galerkin finite element method where
classical derivatives were substituted by weakly defined derivatives
on functions with discontinuity. Optimal order of a priori error
estimates has been observed and established for various weak
Galerkin discretization schemes for second order elliptic equations
\cite{WangYe_PrepSINUM_2011, wy-mixed, mwy-wg-stabilization}.
Numerical implementations of weak Galerkin were discussed in
\cite{MuWangWangYe, mwy-wg-stabilization} for some model problems.
Although the method is still very new, it has already demonstrated many
nice properties in various cases \cite{WangYe_PrepSINUM_2011, WG-biharmonic, wy-mixed, mwy-wg-stabilization}.
One important advantage of the weak Galerkin method is that, with stabilization,
it can be constructed on polytopal meshes, i.e., meshes consisting of arbitrary polygons/polyhedra satisfying certain
shape-regularity conditions.

The weak Galerkin method for the mixed formulation of Darcy flow and
the weak Galerkin method for the Stokes flow have been individually
studied in \cite{wy-mixed} and \cite{wy-Stokes}. It seems that one
only needs to combine these two discretizations together, in order
to derive a discretization for the coupled problem. However, it
turns out that the discretization for the coupled Darcy-Stokes
equations is not that simple. First, due to the interface condition,
the formulation of the Stokes side for the coupled Darcy-Stokes
equation involves the symmetric full stress tensor, and hence is
different from the formulation used in \cite{wy-Stokes}.
Consequently, one needs to use the Korn's inequality and the
discrete Korn's inequality in the analysis, which is one of the
difficulties to be solved in this paper. Because of this, some
discrete spaces we will use for the Stokes side are also different
from the family proposed in \cite{wy-Stokes}. Second, since we need
to prove the discrete inf-sup condition on the entire computational
domain, the discrete spaces we choose for the Darcy side are
completely different from the family proposed in \cite{wy-mixed}.
Finally, we mention that, in order to impose the interface condition
strongly, there must be certain constraints on choosing the Darcy
and Stokes side discretizations. For the weak Galerkin method, this
can be solved by using the same discrete space on edges for both the
Darcy and Stokes side.

Another important issue in discretizing the coupled Darcy-Stokes
equation is whether one can use a unified discretization for both
the Darcy and Stokes sides, which can greatly simplify the numerical
simulation. Previous work on unified discretizations include
conforming finite element methods \cite{Arbogast07}, non-conforming
finite element methods \cite{Karper08}, discontinuous Galerkin
methods \cite{Riviere05,RiviereYotov05}, and $H(div)$ conforming
discontinuous Galerkin methods \cite{KanschatRiviere}. In this
paper, we study several different families of weak Galerkin
discretizations in one single framework for the coupled Darcy-Stokes
equations, in which the choice of different discretization spaces
are controlled by a few parameters denoting the degree of
polynomials. Some of these choices will yield unified discretization
for both the Darcy-Stokes side.

We would also like to mention a few other works on the coupled
Darcy-Stokes equations. The coupled system with the more general
interface condition, the Beavers-Joseph condition, has been studied
in \cite{Cao1,Cao2,Cao3}. Domain decomposition solvers have been
studied by different research groups in
\cite{Cao1,Chen11,DiscacciatiQuarteroni11}. Also, a two-grid solver
has been studied in \cite{MuXu}.

For simplicity, we only consider the two dimensional coupled
Darcy-Stokes equation. It is not hard to extend the analysis into
three dimensions. The paper is organized as follows. Section
\ref{sec:model} is devoted to the introduction of the model problem
and some notations. In Section \ref{sec:wg}, we present the weak
Galerkin discretization for the coupled Darcy-Stokes equation and
prove the existence and uniqueness of the discrete solution. Also in
this section, some technique tools are presented, which will be used
in the error analysis. In Section \ref{sec:error}, error estimates
for the weak Galerkin solution are given. And finally, in Section
\ref{sec:numerical}, numerical results are reported.

%%%%%%%%%%%%%%%%%%%%%%%%%%%%%%%%%%%%%%%%%%%%%%%%%%%%%%%%%%%%%%%%%%
%%%%%%%%%%%%%%%%%%%%%%%%%%%%%%%%%%%%%%%%%%%%%%%%%%%%%%%%%%%%%%%%%%
%%%%%%%%%%%%%%%%%%%%%%%%%%%%%%%%%%%%%%%%%%%%%%%%%%%%%%%%%%%%%%%%%%
\section{Model problem and notation} \label{sec:model}
We follow the same notation system and model problem set-up as in \cite{Chen14}.
For reader's convenience, the details are presented below.

Consider the flow in a domain $\Omega\in \mathbb{R}^2$
consisting of a Stokes sub-region $\Omega_S$ and a porous sub-region
$\Omega_D$.
Denote by $\vu$ the velocity and $p$ the pressure. 
On the Stokes side, the symmetric strain and stress tensors are defined, respectively, by
$D(\vu)=\frac{1}{2}(\nabla\vu+\nabla\vu^T)$ and $\bbT(\vu,\,p) = 2\nu D(\vu) - p I$,
where the given constant $\nu>0$ is the fluid viscosity and $I$ is the identity matrix.
On the Darcy side, denote by $\bbK$ the symmetric positive definite permeability tensor.
Moreover, we assume that $\bbK$ is smooth and uniformly bounded above in $\Omega_D$.

Next we state the coupled Darcy-Stokes equation in $\Omega$.
The flow in the Stokes region is governed by the time-independent Stokes equation,
while the flow in the porous region is governed by the Darcy equation, i.e.,
\begin{equation} \label{eq:eq11}
\begin{aligned}
-\nabla\cdot \bbT(\vu,\, p) &=\vf \qquad\; &&\textrm{in }\Omega_S, \\
\bbK^{-1} \vu + \nabla p &= \vf\quad\qquad &&\textrm{in } \Omega_D, \\
 \nabla\cdot\vu &= g\quad\qquad  &&\textrm{in }\Omega,
\end{aligned}
\end{equation}
where $\vf$ and $g$ are given vector-valued and scalar-valued functions, respectively, in $\Omega$.
Such a coupled system has been studied by many researchers.

To complete the problem, interface conditions and boundary conditions need to be imposed.
Denote by $\Gamma_{SD} = (\partial \Omega_D) \cap (\partial \Omega_S)$ the interface between the Stokes and Darcy regions,
and $\Gamma_S = \partial\Omega_S\backslash \Gamma_{SD}$, $\Gamma_D =
\partial\Omega_D\backslash \Gamma_{SD}$ the outer boundary, as shown in Figure \ref{fig:domain}. 
Following the convention, we denote by $(\vn, \vt)$ the unit
outward normal vector and the unit tangential vector that form a right-hand coordinate system
on the boundary of a given domain.
On the interface $\Gamma_{SD}$, a set of unit normal and unit tangential vectors, $(\hvn, \hvt)$, is specified
such that $\hvn$ points from $\Omega_S$ into $\Omega_D$,
as illustrated in Figure \ref{fig:domain}.

 \begin{figure}[h]
   \begin{center}
   \caption{Domain of the coupled Darcy-Stokes problem. }  \label{fig:domain}
   \includegraphics[width=5cm]{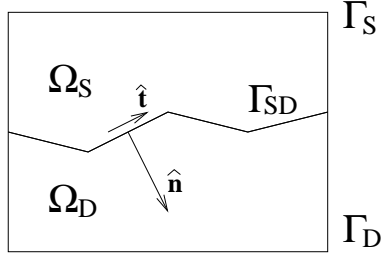}
   \end{center}
 \end{figure}

When necessary, we put $S$ and $D$ in the subscript of $\vu$ and $p$ to distinguish between the Stokes and the Darcy variables,
for example, $\vu_S = \vu|_{\Omega_S}$ and $p_D = p|_{\Omega_D}$.
With the aid of these notations, now we are able to state the boundary and interface conditions.
For simplicity, consider the Dirichlet boundary condition on the Stokes side and the Neumann boundary condition on the Darcy side:
\begin{equation} \label{eq:boundarycondition}
\begin{aligned}
&\vu_S = \vzero &\qquad& \textrm{on }\Gamma_S, \\
&\vu_D\cdot \vn = 0&\qquad& \textrm{on }\Gamma_{D}.
\end{aligned}
\end{equation}
In order to guarantee the uniqueness of the velocity, we assume that $\Gamma_S\neq \emptyset$.
When using the mixed finite element method to discretize system (\ref{eq:eq11}), both boundary conditions in (\ref{eq:boundarycondition}) are essential,
and thus are easier to handle in the rest of the paper. We point out that our analysis can be easily extended to natural boundary conditions,
i.e., Neumann boundary condition on the Stokes side and Dirichlet boundary condition on the Darcy side,
as well as other possible type of boundary conditions.

The interface conditions are defined on $\Gamma_{SD}$ and consist of three parts \cite{BeaversJoseph,Saffman}:
\begin{align}
&\vu_S\cdot\hvn = \vu_D\cdot\hvn, \label{eq:interface1}\\
&-\bbT(\vu_S,\, p_S)\hvn\cdot\hvn = p_D, \label{eq:interface2}\\
&-\bbT(\vu_S,\, p_S)\hvn\cdot \hvt = \mu \bbK^{-1/2} \, \vu_S\cdot \hvt, \label{eq:interface3}
\end{align}
where (\ref{eq:interface3}) is the famous Beavers-Joseph-Saffman condition, in which 
$\mu>0$ is an experimentally determined coefficient.
We assume that $\mu$ is smooth and uniformly bounded both above and away from zero.
Conditions (\ref{eq:interface2}) and (\ref{eq:interface3}) can be combined into one:
\begin{equation} \label{eq:interfaceCondition}
\bbT(\vu_S,\, p_S)\hvn + p_D \hvn + \mu \bbK^{-1/2} \, (\vu_S\cdot \hvt)\hvt = 0\qquad\textrm{on }\Gamma_{SD}.
\end{equation}

By the divergence theorem, the homogeneous boundary condition (\ref{eq:boundarycondition}) requires that 
$g$ satisfies a compatibility condition $\int_{\Omega} g\, dx = 0$.
It is also obvious that the pressure $p$ is unique only up to a constant. Hence we conveniently assume that
$$
\int_{\Omega} p\, dx = 0.
$$

Throughout this paper, we consider the the mixed formulation of problem (\ref{eq:eq11}), which has been studied in details
in \cite{GiraultKanschatRiviere12,Layton03}.
Given a polygon $K$, denote by $H^s(K)$ the
usual Sobolev space equipped with the norm $\|\cdot\|_{s, K}$. For $s =0$,
$H^0(K)$ coincides with the square integrable space $L^2(K)$ and we simply
denote the $L^2$ norm on $K$ by $\|\cdot\|_K$. 
Denote by $(\cdot,\cdot)_{K}$ and
$<\cdot,\cdot>_{K}$ the $L^2$ inner-product and the duality form,
respectively, in $K$. When $K=\Omega$, we suppress the subscript $K$
in the norm and the inner-product, for example, $\|\cdot\|_{s} = \|\cdot\|_{s, \Omega}$, $\|\cdot\| = \|\cdot\|_{\Omega}$,
and $(\cdot,\cdot) = (\cdot,\cdot)_{\Omega}$.
Finally, all the above-defined notations can be easily
extended to vector and tensor spaces, using the usual tensor products.

Define the $H(\div,\, K)$ space and its norm, respectively, by
$$
H(\div,\, K) = \{\vv\in (L^2(K))^2,\, \nabla\cdot\vv\in L^2(K) \},
$$
and
$$
\|\vv\|_{H(\div,\, K)} = (\|\vv\|_K^2 + \|\nabla\cdot\vv\|_K^2)^{1/2}.
$$
The trace of functions in $H(\div,\, K)$ is a subtle issue and has been discussed in
many classical works \cite{GiraultRaviart, LionsMagenes}. Let $\Gamma\subset \partial K$.
We start from defining
$$
H_{0,\Gamma}^1(K) = \{v\in H^1(K), \, v = 0 \textrm{ on }\Gamma \}.
$$
For all $v\in H_{0,\Gamma}^1(K)$, it is well-known by the trace theorem that 
$v|_{\partial K\backslash \Gamma}\in H_{00}^{1/2}(\partial K\backslash \Gamma)$,
where $H_{00}^{1/2}(\partial K\backslash \Gamma)$ is a subspace of $H^{1/2}(\partial K\backslash \Gamma)$
consisting of functions that can be extended by $0$ to $H^{1/2}(\partial K)$. 
Readers can refer to \cite{GiraultRaviart, LionsMagenes} for more details.
For any function $\vv\in H(\div,\, K)$, by the trace theorem one has 
$\vv\cdot\vn|_{\partial K} \in H^{-1/2}(\partial K)$. However, $\vv\cdot\vn$ may not be well-defined on a subset of $\partial K$. 
One needs to use the dual space of $H_{00}^{1/2}$ in order to obtain a rigorous definition of the trace.
Define
$$
H_{0,\Gamma}(\div,\, K) = \{\vv\in H(\div,\,K),\, \vv\cdot\vn = 0 \textrm{ on }\Gamma \},
$$
where $\vv\cdot\vn = 0$ is in the sense of $\vv\cdot\vn \in (H_{00}^{1/2}(\Gamma))^*$.
%for all $\vv\in H_{0,\Gamma}(\div,\, K)$, we have $\vv\cdot\vn|_{\partial K\backslash \Gamma}\in (H_{00}^{1/2}(\partial K\backslash \Gamma))^*$.
When $\Gamma = \partial K$, we simply denote $H_{0,\partial K}(\div,\, K) = H_{0}(\div,\, K)$
and $H_{0,\partial K}^1(K) = H_0^1(K)$.

Now we are able to introduce the mixed variational formulation for system (\ref{eq:eq11})-(\ref{eq:interface3}). 
To this end, we start from defining the spaces for the velocity and the pressure, respectively, by
$$
\begin{aligned}
\vV &= \{\vv\in H_0(\div, \Omega)\, |\,\vv_S\in H^1(\Omega_S)^2 \textrm{ and } \vv|_{\Gamma_S} = \vzero \}, \\
\Psi &= L_0^2(\Omega) \triangleq \{q\in L^2(\Omega)\, |\, \int_\Omega q\, dx = 0\}.
\end{aligned}
$$
It is not hard to see that $\vV$ and $\Psi$, equipped with the norms $(\|\vv\|_{1,\Omega_S}^2 + \|\vv\|_{H(\div,\Omega_D)}^2)^{1/2}$
and $\|q\|$ respectively, are both Hilbert spaces.
An important property of $\vV$ is that, all functions in $\vV$ explicitly satisfy the interface condition (\ref{eq:interface1}),
according to the properties \cite{Brezzi} of $H(\div,\,\Omega)$.
%For convenience, denote by $\vV_S$ and $\vV_D$ the restrictions of $\vV$ on $\Omega_S$ and $\Omega_D$.
%Furthermore, $\vv_S|_{\Gamma_{SD}} \in H_{00}^{1/2}(\Gamma_{SD})^2$ for all $\vv\in \vV$.

Define bilinear forms $a(\cdot, \cdot):\: \vV\times \vV\to \bbR$ and $b(\cdot, \cdot):\: \vV\times \Psi\to \bbR$ by
$$
\begin{aligned}
a(\vu, \vv) &= a_S(\vu, \vv) + a_D(\vu, \vv) + a_I(\vu, \vv),\\
b(\vv, q) &= -(\nabla\cdot \vv, q),
\end{aligned}
$$
where
$$
\begin{aligned}
a_S(\vu, \vv) & = 2\nu (D(\vu), D(\vv))_{\Omega_S}, \\
a_D(\vu, \vv) &= (\bbK^{-1} \vu, \vv)_{\Omega_D}, \\
a_I(\vu, \vv) &= <\mu\bbK^{-1/2}\vu_s\cdot\hvt, \vv_s\cdot\hvt>_{\Gamma_{SD}}.
\end{aligned}
$$
The mixed variational formulation of the coupled Darcy-Stokes equation can be written as:
{\it Find $(\vu, \, p)\in \vV\times \Psi$ such that}
\begin{equation} \label{eq:mixedweakformulation}
\begin{cases}
a(\vu, \vv) + b(\vv, p) = (\vf, \vv) \quad& \textrm{for all } \vv\in \vV, \\
b(\vu, q) = -(g, q)\quad& \textrm{for all } q\in \Psi.
\end{cases}
\end{equation}
In \cite{Layton03}, Layton, Schieweck and Yotov has proved the equivalence between (\ref{eq:mixedweakformulation}) and (\ref{eq:eq11})-(\ref{eq:interface3}),
as well as the existence and uniqueness of the solution to (\ref{eq:mixedweakformulation}).

\section{Weak Galerkin discretization} \label{sec:wg}
In this section, we discuss the weak Galerkin discretization for the coupled Darcy-Stokes problem (\ref{eq:eq11})-(\ref{eq:interface3}).
For simplicity of notation, throughout the paper, we use ``$\lesssim$''
to denote ``less than or equal to up to a general constant
independent of the mesh size or functions appearing in the
inequality''.
But "$\lesssim$" may depend on $\nu$, $\bbK$, $\mu$, $\Omega$ or $\Gamma_{SD}$.

Let $\mathcal{T}_h$ be a polygonal mesh defined on $\Omega$
satisfying the shape regularity conditions proposed in
\cite{MuWangWang,wy-mixed}, in order to guarantee the existence of
the usual trace inequality, inverse inequality, and the
approximability of polynomials on polygons. We require that
$\mathcal{T}_h$ be aligned with $\Gamma_{SD}$. For each polygon
$K\in \mathcal{T}_h$, denote by $K_0$ and $\partial K$ the interior
and the boundary of $K$, respectively. Also, denote by $h_K$ the
diameter of the element $K$, and set
$h=\max_{K\in\mathcal{T}_h}h_K$. Denote by $\mathcal{T}_h^S$ and
$\mathcal{T}_h^D$ the restriction of $\T_h$ in $\Omega_S$ and
$\Omega_D$, respectively.

Denote by ${\E}_h$  the set of all edges in ${\T}_h$.
For each edge $e\in \E_h$, denote by $h_e$ its length.
Let $\mathcal{E}_h^{SD}$ be the set of all edges in $\mathcal{T}_h\cap \Gamma_{SD}$,
and let $\mathcal{E}_h^S$, $\mathcal{E}_h^D$ be the set of all edges in
$\mathcal{T}_h\cap (\Omega_S\cup\Gamma_S)$, $\mathcal{T}_h\cap (\Omega_D\cup\Gamma_D)$, respectively.
We also denote $\circES$ and $\circED$ to be the set of edges interior to $\Omega_S$ and $\Omega_D$, respectively.

Let $j$ be a non-negative integer. On each $K\in \mathcal{T}_h$,
denote by $P_j(K_0)$ or $P_j(K)$ the set of polynomials with degree
less than or equal to $j$. Likewise, on each $e\in \E_h$, $P_j(e)$
is the set of polynomials of degree no more than $j$. Following
\cite{wy-mixed,wy-Stokes}, we define the weak Galerkin spaces:
$$
\begin{aligned}
\vV_h &= \{\vv=\{\vv_0,\vv_b\}:\,&& \vv_0|_{K_0} \in
[P_{\alpha_S}(K_0)]^2\textrm{ for }K\in \T_h^S, \\
    &&&                             \vv_b|_e \in [P_{\beta}(e)]^2\textrm{ for }e\in \E_h^S\cup\E_h^{SD}, \\
    &&& \vv_0|_{K_0} \in [P_{\alpha_D}(K_0)]^2\textrm{ for }K\in
    \T_h^D,\\
    &&&                             \vv_b|_e = v_b \vn_e \textrm{ where }v_b \in P_{\beta}(e)\textrm{ for }e\in \E_h^D, \\
    &&& \vv_b|_e = 0 \textrm{ for }e\in \E_h\cap \partial\Omega\},
\end{aligned}
$$
where $\alpha_S$, $\alpha_D$ and $\beta$ are non-negative integers, $\vn_e$ is a prescribed normal direction for edge $e\in \E_h^D$, and
$$
\Psi_h = \{q\in L_0^2(\Omega) :\, q|_K\in P_{\gamma_S}(K)\textrm{ for } K\in \T_h^S \textrm{ and }
  q|_K\in P_{\gamma_D}(K)\textrm{ for } K\in \T_h^D \},
$$
where $\gamma_S$ and $\gamma_D$ are non-negative integers.
Moreover, assume that
\begin{equation} \label{eq:parameters}
\begin{aligned}
\beta-1\le\gamma_S &\le \beta\le \alpha_S\le\beta+1, \\
\beta-1\le\gamma_D &\le \beta= \alpha_D, \\
\alpha_S &\le \gamma_S+1.&&
\end{aligned}
\end{equation}
Later we shall discuss more about the choice of parameters
$\alpha_S$, $\alpha_D$, $\beta$, $\gamma_S$, and $\gamma_D$. But let
use first give two examples that satisfy (\ref{eq:parameters}), both
providing unified discretizations for both the Darcy and Stokes
sides:
\begin{description}
\item[Example 1] Set $\alpha_S=\alpha_D=\beta=\gamma_S=\gamma_D = j$ where $j\ge
1$;
\item[Example 2] Set $\alpha_S = \alpha_D=\beta = j$ and $\gamma_S = \gamma_D = j-1$, where $j\ge 1$.
\end{description}

\begin{remark}
Condition (\ref{eq:parameters}) is derived from certain constraints
on constructing weak Galerkin spaces, which will become clear after
defining the weak gradient, the weak divergence and the
well-posedness of the discrete system. Indeed, a minimum set of
conditions on the parameters looks like
\begin{equation}  \label{eq:parameters-min}
\begin{aligned}
\beta-1\le\gamma_S &\le \beta\le \alpha_S\le\beta+1, \\
\beta-1\le\gamma_D &\le \beta\le \alpha_D\le\beta+1, \\
\alpha_S \le \gamma_S+1,&\qquad \alpha_D \le \gamma_D+1,
\end{aligned}
\end{equation}
which is more general than Condition (\ref{eq:parameters}). However,
when performing the error analysis we realized that, by enforcing
$\alpha_D = \beta$, the theoretical error estimate in terms of
$\gamma_D$ is one order higher than in the case of $\alpha_D >
\beta$. Hence in this paper we shall focus on the case $\alpha_D =
\beta$. Note that Condition (\ref{eq:parameters-min}) combined with
$\alpha_D = \beta$ gives exactly Condition (\ref{eq:parameters}).
\end{remark}

Note that for both the Stokes side and the Darcy side, we
deliberately set $\vv_b$ to have the same polynomial degree, which
ensures seamless transition on the Darcy-Stokes interface. The
spaces defined above are different from the spaces introduced in
\cite{wy-mixed,wy-Stokes}, which are designed for the Darcy flow and
for the Stokes equations individually. This is because we have
found, while working on the theoretical analysis, that the spaces in
\cite{wy-mixed,wy-Stokes} may not be suitable for discretizing the
coupled Darcy-Stokes equation. Therefore, we have to construct a new
set of spaces $\vV_h$ and $\Psi_h$.

Next, we define the weak gradient and the weak divergence on
$\vV_h$. Both of them are defined element-wisely. For each $K\in
\T_h$ and $\vv=\{\vv_0,\vv_b\}$, define the weak gradient
$\nabla_w\vv|_K \in [P_{\beta}(K)]^{2\times 2}$ and weak divergence
$\nabla_w\cdot\vv \in P_{\beta}(K)$, respectively, by
\begin{align}
(\nabla_w \vv, \tau)_K &= -(\vv_0, \nabla\cdot \tau)_K + <\vv_b, \tau\vn>_{\partial K} \quad&&\Forall \tau \in [P_{\beta}(K)]^{2\times 2},
  \label{eq:weakgrad}\\
(\nabla_w\cdot \vv, q)_K &= -(\vv_0, \nabla q)_K + <\vv_b\cdot\vn, q>_{\partial K} \quad &&\Forall q\in P_{\beta}(K).
  \label{eq:weakdiv}
\end{align}
Note that the weak gradient is only needed on $K\in\T_h^S$, while
the weak divergence is needed on both $K\in\T_h^S$ and $K\in\T_h^D$.
By the definition of the space $\vV_h$, when an edge $e$ of $K\in
\T_h^D$ lies in $\E_h^D$, we have $\vv_b = v_b \vn_e$ where $v_b\in
P_{\beta}(e)$ on this edge; and when an edge $e$ of $K\in \T_h^D$
lies in $\E_h^{SD}$, we have $\vv_b\in [P_{\beta}(e)]^2$ on this
edge. In both cases, $\vv_b\cdot\vn \in P_{\beta}(e)$ and thus the
definition of the weak divergence in (\ref{eq:weakdiv}) is
consistent on all $K\in \T_h$.

Denote
$$
D_w(\vv) = \frac{1}{2} \left(\nabla_w \vv + (\nabla_w\vv)^T\right).
$$
On each $K\in \T_h$, denote by $Q_0$ the $L^2$ projection onto
$(P_{\alpha_{S}}(K))^2$ or $(P_{\alpha_{D}}(K))^2$, depending on
whether $K$ is in $\T_h^S$ or $T_h^D$. On each $e\in\E_h$, denote by
$Q_b$ the $L^2$ projection onto $(P_{\beta}(e))^2$ or
$P_{\beta}(e)\vn$, depending on whether $e$ is on the Stokes side or
the Darcy side. On the Stokes side including the interface, $Q_b$
operates on both components of the velocity, while on the Darcy
side, it only operates on the normal components of the velocity. On
the interface $\Gamma_{SD}$, the normal component of the velocity is
continuous and thus $Q_b$ transits naturally between the Darcy and
the Stokes subdomains. Combining these local projections together,
we can define an $L^2$ projection $Q_h=\{Q_0,Q_b\}$ onto $\vV_h$.
Similarly, denote by $\bbQ_h$ the $L^2$ projection onto $\Psi_h$.
Now we can define the bilinear forms
$$
\begin{aligned}
a_h(\vu, \vv) &= a_{h,S}(\vu, \vv) + a_{h,D}(\vu, \vv) + a_I(\vu, \vv)\qquad &&\textrm{for }\vu, \vv\in \vV_h, \\
b_h(\vv, q) &= -(\nabla_w\cdot \vv, q) \qquad &&\textrm{for } \vv\in
\vV_h\textrm{ and }q\in \Psi_h,
\end{aligned}
$$
where
$$
\begin{aligned}
a_S(\vu, \vv) & = 2\nu (D_w(\vu), D_w(\vv))_{\Omega_S}
   + \rho_S \sum_{K\in \T_h^S} h_K^{-1} <Q_b\vu_0-\vu_b, Q_b\vv_0-\vv_b>_{\partial K}, \\
a_D(\vu, \vv) &= (\bbK^{-1} \vu_0, \vv_0)_{\Omega_D}
  + \rho_D \sum_{K\in\T_h^D} h_K^{-1} <(\vu_0-\vu_b)\cdot\vn, (\vv_0-\vv_b)\cdot\vn>_{\partial K} , \\
a_I(\vu, \vv) &= <\mu\bbK^{-1/2}\vu_b\cdot\hvt, \vv_b\cdot\hvt>_{\Gamma_{SD}},
\end{aligned}
$$
in which $\rho_S$ and $\rho_D$ are positive constants. One can view
$\rho_S$ and $\rho_D$ as stabilization parameters, but the good news
is that the weak Galerkin method does not depend on these parameters
as the discontinuous Galerkin method does. One can simply set
$\rho_S=\rho_D=1$.
%We also point out that the stabilization term for
%the Darcy region is different from the stabilization used in
%\cite{wy-mixed} for the mixed formulation of Poisson equations. The
%stabilization term given in \cite{wy-mixed} uses $h_K
%<(\vu_0-\vu_b)\cdot\vn, (\vv_0-\vv_b)\cdot\vn>_{\partial K}$ instead
%of $h_K^{-1} <(\vu_0-\vu_b)\cdot\vn,
%(\vv_0-\vv_b)\cdot\vn>_{\partial K}$.

The weak Galerkin formulation for the Darcy-Stokes flow can now be
written as: find $\vu\in\vV_h$ and $p\in \Psi_h$ such that
\begin{equation} \label{eq:weakGalerkinformulation}
\begin{cases}
a_h(\vu, \vv) + b_h(\vv, p) = (\vf, \vv_0) \quad& \textrm{for all } \vv\in \vV_h, \\
b_h(\vu, q) = -(g, q)\quad& \textrm{for all } q\in \Psi_h.
\end{cases}
\end{equation}
We shall analyze the well-posedness and approximation properties of the weak Galerkin discretization
(\ref{eq:weakGalerkinformulation}).

\begin{remark}
The bilinear form $a_h(\vu,\vv)$ contains a stabilization part, and for convenience, we denote it by
$$
\begin{aligned}
s(\vu,\vv) &= \rho_S \sum_{K\in \T_h^S} h_K^{-1} <Q_b\vu_0-\vu_b,
Q_b\vv_0-\vv_b>_{\partial K} \\
  &\qquad + \rho_D \sum_{K\in\T_h^D} h_K^{-1} <(\vu_0-\vu_b)\cdot\vn, (\vv_0-\vv_b)\cdot\vn>_{\partial K} .
  \end{aligned}
$$
\end{remark}
\smallskip

%%%%%%%%%%%%%%%%%%%%%%%%%%%%%%%%%%%%%%%%%%%%%%%%%%%%%%%%%%%%%%%%%%%%%%%%%%%%%%%%%
\subsection{Discrete norm}
Define a discrete norm on $\vV_h$ by
$$
\begin{aligned}
\|\vv\|_{\vV_h} &= \bigg(2\nu\|D_w(\vv)\|_{\Omega_S}^2 + \rho_S\sum_{K\in\T_h^S} h_K^{-1}\|Q_b\vv_0-\vv_b\|_{\partial K}^2 \\
&\qquad + \|\bbK^{-1/2}\vv_0\|_{\Omega_D}^2 + \rho_D\sum_{K\in\T_h^D} h_K^{-1} \|(\vv_0-\vv_b)\cdot\vn\|_{\partial K}^2 \\
&\qquad + \|\nabla_w\cdot\vv\|_{\Omega}^2 + \|\mu^{1/2}\bbK^{-1/4}\vv_b\cdot\hvt\|_{\Gamma_{SD}}^2 \bigg)^{1/2}.
\end{aligned}
$$
It is obvious that $a_h(\vv,\vv) = \|\vv\|_{\vV_h}^2$. The discrete
norm on $\Psi_h$ inherits the norm on $\Psi$, which is just the
$L^2$ norm. Denote $\3bar\cdot\3bar = \|\cdot\|_{\vV_h\times
\Psi_h}$. We shall prove that $\|\cdot\|_{\vV_h}$ is a well-defined
norm for certain choices of parameters $\alpha_S$, $\alpha_D$,
$\beta$, $\gamma_S$ and $\gamma_D$. To this end, we only need to
show that $\|\vv\|_{\vV_h}=0$ implies $\vv \equiv 0$ for $\vv\in
\vV_h$. By definition, $\|\vv\|_{\vV_h}=0$ indicates that
\begin{align}
D_w(\vv) &= \vzero \textrm{ on } K\in \T_h^S, & Q_b\vv_0 - \vv_b &= \vzero \textrm{ on } e\in \E_h^S\cup\E_h^{SD},
  \label{eq:well-posedness1-S}\\
\vv_0 &= \vzero \textrm{ on } K\in \T_h^D, & (\vv_0-\vv_b)\cdot\vn &= 0 \textrm{ on } e\in \E_h^D\cup\E_h^{SD},
  \label{eq:well-posedness1-D}\\
&& \vv_b\cdot \hvt &= 0 \textrm{ on } e\in \E_h^{SD}. \label{eq:well-posedness1-I}
\end{align}
By examining these equations and depending on whether $\alpha_S = \beta$ or $\alpha_S = \beta+1$, we have the following results.

\begin{lemma}  \label{lem:vV-h-norm-1}
If $\alpha_S = \beta$, then $\|\cdot\|_{\vV_h}$ is a well-defined norm on $\vV_h$.
\end{lemma}
{\bf Proof}
From (\ref{eq:well-posedness1-D}), it is not hard to see that $\vv_0$ and $\vv_b$ vanishes on all
$K\in \T_h^D$ and $e\in \E_h^D$. Combining (\ref{eq:well-posedness1-D}) and (\ref{eq:well-posedness1-I}),
we also know that $\vv_b$ vanishes on $\Gamma_{SD}$.

On $K\in \T_h^S$, using the definition (\ref{eq:weakgrad}) gives
$$
\begin{aligned}
((\nabla_w \vv)^T, \tau)_K = (\nabla_w \vv, \tau^T)_K =
 -(\vv_0, \nabla\cdot \tau^T)_K + &<\vv_b, \tau^T\vn>_{\partial K} \\
 & \quad\Forall \tau \in [P_{\beta}(K)]^{2\times 2},
 \end{aligned}
$$
which, combined with (\ref{eq:weakgrad}), implies that
$$
2(D_w(\vv), \tau)_K = -(\vv_0, \nabla\cdot (\tau+\tau^T))_K + <\vv_b, (\tau+\tau^T)\vn>_{\partial K} \quad\Forall \tau \in [P_{\beta}(K)]^{2\times 2}.
$$
Therefore, $D_w(\vv) = \vzero$ on $K\in \T_h^S$ implies that for all symmetric $\tau\in [P_{\beta}(K)]^{2\times 2}$,
$$
\begin{aligned}
0 &= -(\vv_0, \nabla\cdot \tau)_K + <\vv_b, \tau\vn>_{\partial K} \\
    &= (\nabla \vv_0, \tau)_K - <\vv_0 - \vv_b, \tau\vn>_{\partial K} \\
    &= (D(\vv_0), \tau)_K - <Q_b \vv_0 - \vv_b, \tau\vn>_{\partial K} \\
    &= (D(\vv_0), \tau)_K,
\end{aligned}
$$
where the last step follows from (\ref{eq:well-posedness1-S}). Hence
we have $D(\vv_0) \equiv \vzero$ on all $K\in \T_h^S$. By the
definition of $D(\cdot)$, this in turn implies that $\vv_0|_K \in
RM$ where $RM = span\{\begin{bmatrix}1\\0\end{bmatrix},
\begin{bmatrix}0\\1\end{bmatrix},
\begin{bmatrix}-y\\x\end{bmatrix}\}$ denotes the space of rigid body
motions. If we can further show that $\vv_0$ is continuous on the
entire $\Omega_S$, then by the definition of $RM$ it is not hard to
see that $\vv_0|_{\Omega_S} \in RM$.

Now, since $\alpha_S = \beta$, by (\ref{eq:well-posedness1-S}) we known that $\vv_0$ must be continuous in the entire
$\Omega_S$. Therefore $\vv_0|_{\Omega_S} \in RM$.
Note that $\vv_b$ vanishes on $\partial\Omega_S$, again by (\ref{eq:well-posedness1-S}), $\vv_0$ must also vanish on
$\partial\Omega_S$, which implies that $\vv_0\equiv \vzero$ in $\Omega_S$.
Consequently, $\vv_b\equiv \vzero$ on $e\in \E_h^S\cup\E_h^{SD}$. This completes the proof of the lemma.
$\Box$
\medskip

For $\alpha_S = \beta+1$, the situation is more complicated. Using
the same argument as in the proof of Lemma \ref{lem:vV-h-norm-1},
for $\alpha_S = \beta+1$ we can still prove that $\vv_0=\vv_b \equiv
\vzero$ in $\Omega_D$ and $\vv_0|_K \in RM$ for $K\in \T_h^S$. Now,
$Q_b$ on the Stokes side is no longer an identity operator, and
hence Equation (\ref{eq:well-posedness1-S}) only implies that
$Q_b\vv_0$ is continuous across the edges in $\Omega_S$, while
$\vv_0$ is not necessarily continuous. We need to consider two cases
separately, the case $\beta\ge 1$ and the case $\beta=0$.

For $\beta\ge 1$ and $\vv_0|_K\in RM\subset [P_1(K)]^2$, we can
still get $Q_b\vv_0 = \vv_0$ on edges, which together with the
continuity of $Q_b\vv_0$ across edges implies that $\vv_0$ is
continuous in the entire $\Omega_S$. Thus $\vv_0|_{\Omega_S} \in
RM$. The rest is similar to the proof of Lemma \ref{lem:vV-h-norm-1}
and is summarized in the following lemma:

\begin{lemma} \label{lem:vV-h-norm-2}
If $\alpha_S = \beta+1$ and $\beta\ge 1$, then $\|\cdot\|_{\vV_h}$ is a well-defined norm on $\vV_h$.
\end{lemma}

Finally, we consider the case of $\alpha_S = \beta+1$ and $\beta=0$.
When $\beta=0$, Equation (\ref{eq:well-posedness1-S}) implies that
$\vv_0$ is continuous only at the center of internal edges in
$\T_h^S$. On boundary edges $e\in \E_h\cap\partial\Omega_S$, the
value of $\vv_0$ at the center is equal to the value of $\vv_b$,
which is $\vzero$. If $\T_h^S$ is a triangular mesh, clearly $\vv_0$
can be viewed as in the space of the lowest order Crouzeix-Raviart
non-conforming finite element defined on the triangular mesh
$\T_h^S$ with zero boundary condition. In this case, Falk
\cite{Falk91} has shown, using dimension counting, that $\vv_0\in
RM$ on all $K\in \T_h^S$ is not enough to guarantee $\vv_0 \equiv
\vzero$. This is just the famous result that the lowest order
Crouzeix-Raviart non-conforming finite element does not satisfy the
discrete Korn's inequality.

However, the situation can be different when the mesh contains general polygons.
To analyze this, we first introduce a few tools.

\begin{lemma} \label{lem:A1}
Let $\vv_0 \in RM$ on a polygon $K$. If $\vv_0$ vanishes on two different points in $K$, then $\vv_0\equiv \vzero$.
\end{lemma}
{\bf Proof} Any $\vv_0\in RM$ can be written as $\begin{bmatrix}a-cy\\b+cx\end{bmatrix}$.
Two different points either have different $x$-coordinates or $y$-coordinates. If $\vv_0$ vanishes
on both points, it is not hard to see that $c$ must be $0$. Consequently, $a$ and $b$ must also be $0$.
This completes the proof of the lemma.
$\Box$

\begin{algorithm} \label{alg:meshsweeping}
We start from setting all polygons in $\T_h^S$ black.
\begin{enumerate}
\item For all $K\in\T_h^S$, set $K$ white if $K$ has two different edges lying on $\partial\Omega_S$;
\item For all black polygons,
  \begin{itemize}
  \item Set polygon $K$ white if $K$ has one edge lying on $\partial\Omega_S$ and shares another edge with a white polygon;
  \item Set polygon $K$ white if $K$ shares two edges with other white polygons;
  \end{itemize}
\item Repeat Step 2 until there is no new coloring.
\end{enumerate}
\end{algorithm}

\begin{definition}
Mesh $\T_h^S$ is colorable if Algorithm \ref{alg:meshsweeping} will turn all polygons in $\T_h^S$ white.
Otherwise, it is not colorable.
\end{definition}

Clearly, a simple example of colorable mesh is the rectangular grid.
We know that $\vv_0$ vanishes at the center of each $e\in \E_h\cap
\partial \Omega_S$. Thus, the property stated in Lemma \ref{lem:A1}
can propagate to all white polygons generated by Algorithm
\ref{alg:meshsweeping}. This leads to the following conclusion:

\begin{theorem}
When $\beta = 0$ and $\alpha_S = 1$, if $\T_h^S$ is colorable, then $\|\cdot\|_{\vV_h}$ is a well-defined norm.
\end{theorem}

\begin{remark}
In the rest of this paper, when $\beta = 0$ and $\alpha_S = 1$, we always assume that the mesh is colorable.
In other words, $\|\cdot\|_{\vV_h}$ is always well-defined in this paper.
\end{remark}
\smallskip

%%%%%%%%%%%%%%%%%%%%%%%%%%%%%%%%%%%%%%%%%%%%%%%%%%%%%%%%%%%%%%%%%%%%%%%%%%%%%%%%%%%%%%%%%
\subsection{Existence and uniqueness of the discrete solution}
Given that $\|\cdot\|_{\vV_h}$ is a well-defined norm in $\vV_h$, we can easily derive the existence and uniqueness
of the solution to System (\ref{eq:weakGalerkinformulation}).

\begin{theorem} \label{thm:well-posedness}
System (\ref{eq:weakGalerkinformulation}) admits a unique solution.
\end{theorem}
{\bf Proof} For discrete problems, uniqueness implies existence of
the solution. Therefore we only need to prove that when $\vf\equiv
\vzero$ and $q\equiv 0$, the solution to
(\ref{eq:weakGalerkinformulation}) is exactly zero. By setting $\vv
= \vu$ and $q = p$ and subtracting the two equations in
(\ref{eq:weakGalerkinformulation}), one has $0=a_h(\vu, \vu) =
\|\vu\|_{\vV_h}^2$, which clearly implies $\vu\equiv\vzero$.

Next, by using $0=b_h(\vv, p) = -(\nabla_w\cdot\vv, p)$ for all
$\vv\in \vV_h$, we will show that $p\equiv 0$ on the entire
$\Omega$. Since $\gamma_S\le\beta$ and $\gamma_D\le \beta$, by
(\ref{eq:weakdiv}) we have for all $\vv\in \vV_h$
\begin{equation} \label{eq:well-posedness-p}
0 = (\nabla_w\cdot\vv, p) = \sum_{K \in\T_h} \left(-(\vv_0, \nabla p)_K + <\vv_b\cdot\vn, p>_{\partial K}\right).
\end{equation}
By setting $\vv_b = \vzero$ on all edges and let $\vv_0$ vanish on
all except for one polygon in $\T_h$, one gets
$$
(\vv_0, \nabla p)_K = 0\textrm{ for all } \vv\in \vV_h \textrm{ and }K\in \T_h.
$$
Note that for polygons lying on either the Darcy side or the Stokes
side, according to the definition of $\vV_h$ and $\Psi_h$, the space
of $\nabla p$ is always contained in the space of $\vv_0$. Thus we
conclude that $\nabla p|_K = \vzero$ for all $K\in \T_h$, which
implies that $p$ is piecewise constant.

Next, let $\vv_0\equiv \vzero$ in Equation (\ref{eq:well-posedness-p}), we have for all $\vv_b$
\begin{equation} \label{eq:well-posedness-p2}
0 = \sum_{K \in\T_h} <\vv_b\cdot\vn, p>_{\partial K} = \sum_{e\in \E_h} <\vv_b\cdot\vn, [p]>_e,
\end{equation}
where $[p]$ denotes the jump of $p$ on edge $e$. On boundary edges,
$[p]$ is just defined to be the one-sided value of $p$. Note that
the summation in (\ref{eq:well-posedness-p2}) does not need to be
distinguished on the Darcy or the Stokes side, because $\vv_b\cdot
\vn$ for both the Darcy side and the Stokes side belongs to the same
discrete space. Combining Equation (\ref{eq:well-posedness-p2}) with
the fact that $p$ is piecewise constant, we conclude that $p$ must
be a constant on the entire $\Omega$. And since $p\in
L_0^2(\Omega)$, thus $p\equiv 0$. This completes the proof of the
theorem. $\Box$
\smallskip

%%%%%%%%%%%%%%%%%%%%%%%%%%%%%%%%%%%%%%%%%%%%%%%%%%%%%%%%%%%%%%%%%%
\subsection{A few technique tools}
In this subsection we introduce a few technique tools that will be
used in the error analysis of the weak Galerkin approximation.
First, for any $K\in \T_h$ and $e$ being an edge of $K$, the
following trace inequality is known \cite{MuWangWang,wy-mixed}
\begin{equation}\label{eq:boundary2interior}
\|\phi\|_{e}^2 \lesssim  h_K^{-1} \|\phi\|_K^2 + h_K |\phi|_{1,K}^2,
\end{equation}
for all $\phi\in H^1(K)$. Unlike the usual trace inequalities, the
inequality (\ref{eq:boundary2interior}) is prove on polytopal meshes
satisfying certain shape regularity conditions
\cite{MuWangWang,wy-mixed}. For such meshes, the inverse inequality
and the approximation property of $L^2$ projections onto polynomial
spaces have also been proved \cite{MuWangWang,wy-mixed}. These
inequalities have the same form as their counterparts on triangular
and rectangular meshes. In the rest of this paper, we will use them
directly, without special mentioning.

On each $K\in \T_h$, denote by $\Pi_h$ and $\pi_h$ the $L^2$
projections onto $[P_{\beta}(K)]^{2\times 2}$ and $P_{\beta}(K)$,
respectively. And on the entire $\Omega$, we use the same notations,
$\Pi_h$ and $\pi_h$, to denote the combination of all local
projections. Then one has

\begin{lemma} \label{lem:projections}
The projection operators satisfy
$$
\begin{aligned}
\nabla_w(Q_h \bv) &= \Pi_h (\nabla \bv) \qquad &&\textrm{for all }\vv\in [H^1(\Omega)]^2, \\
\nabla_w\cdot (Q_h\vv) &= \pi_h (\nabla\cdot \bv) \qquad && \textrm{for all }\vv\in H(\div,\Omega).
\end{aligned}
$$
\end{lemma}
{\bf Proof}
By (\ref{eq:parameters}), (\ref{eq:weakgrad}), the definitions of $Q_h$ and $\Pi_h$,
clearly for all $\tau\in [P_{\beta}(K)]^{2\times 2}$ and $K\in \T_h$,
$$
\begin{aligned}
(\nabla_w(Q_h \vv), \tau)_K &= -(Q_0\vv, \nabla\cdot\tau)_K + <Q_b\vv, \tau\vn>_{\partial K} \\
&= -(\vv, \nabla\cdot\tau)_K + <\vv, \tau\vn>_{\partial K} \\
&= (\nabla \vv, \tau)_K = (\Pi_h (\nabla\vv), \tau)_K.
\end{aligned}
$$
Similarly, for all $q\in P_{\beta}(K)$ and $K\in \T_h$,
$$
\begin{aligned}
(\nabla_w\cdot (Q_h\vv), q)_K &= -(Q_0\vv, \nabla q)_K + <(Q_b\vv)\cdot\vn, q>_{\partial K} \\
&= -(\vv, \nabla q)_K + <\vv\cdot\vn, q>_{\partial K} \\
&= (\nabla\cdot \vv, q)_K = (\pi_h \nabla\cdot \vv, q)_K.
\end{aligned}
$$
This completes the proof of the lemma.
$\Box$
\medskip

Next we prove the discrete inf-sup condition:
\begin{lemma} \label{lem:inf-sup}
For all $q\in \Psi_h$, one has
$$
\sup_{\vv\in \vV_h} \frac{(\nabla_w\cdot\vv, q)}{\|\vv\|_{\vV_h}} \gtrsim \|q\|.
$$
\end{lemma}
{\bf Proof}
It is well known that for all $q\in \Psi_h\subset L_0^2(\Omega)$, there exists a $\vw\in [H_0^1(\Omega)]^2$ such that
$\nabla\cdot\vw = q$ and $\|\vw\|_1\lesssim \|q\|$. Define $\vv = Q_h\vw$, then by Lemma \ref{lem:projections},
and the facts that $\gamma_S\le \beta$, $\gamma_D\le \beta$, we have
$$
(\nabla_w\cdot\vv, q) = (\pi_h (\nabla\cdot \vw), q) = (\nabla\cdot \vw, q) = \|q\|^2.
$$
Now, we only need to prove that
$$
\|\vv\|_{\vV_h} \lesssim \|\vw\|_1.
$$

Note that
$$
\begin{aligned}
\|\vv\|_{\vV_h} &= \bigg(2\nu\|D_w(Q_h\vw)\|_{\Omega_S}^2 + \rho_S\sum_{K\in\T_h^S} h_K^{-1}\|Q_b(Q_0\vw)-Q_b\vw\|_{\partial K}^2 \\
&\qquad + \|\bbK^{-1/2}Q_0\vw\|_{\Omega_D}^2 + \rho_D\sum_{K\in\T_h^D} h_K^{-1} \|(Q_0\vw-Q_b\vw)\cdot\vn\|_{\partial K}^2 \\
&\qquad + \|\nabla_w\cdot(Q_h\vw)\|_{\Omega}^2 + \|\mu^{1/2}\bbK^{-1/4}(Q_b\vw)\cdot\hvt\|_{\Gamma_{SD}}^2 \bigg)^{1/2}.
\end{aligned}
$$
We will check the terms in the above equation one-by-one. First, by
Lemma \ref{lem:projections},
$$
\begin{aligned}
\|D_w(Q_h\vw)\|_{\Omega_S}^2 &= \|\frac{1}{2}(\nabla_w(Q_h\vw)+\nabla_w(Q_h\vw)^T)\|_{\Omega_S}^2
= \|\frac{1}{2}(\Pi_h\nabla\vw + (\Pi_h\nabla\vw)^T)\|_{\Omega_S}^2 \\
&= \|\Pi_h D(\vw)\|_{\Omega_S}^2 \lesssim \|\vw\|_{1,\Omega_S}^2,
\end{aligned}
$$
and obviously $\|\bbK^{-1/2}Q_0\vw\|_{\Omega_D} \lesssim \|\vw\|_{\Omega_D}$.
Next, by using the properties of $Q_h$ and the inequality (\ref{eq:boundary2interior}), we have for $K\in\T_h^S$,
$$
\begin{aligned}
h_K^{-1}\|Q_b(Q_0\vw)-Q_b\vw\|_{\partial K}^2 &= h_K^{-1}\|Q_b(Q_0\vw-\vw)\|_{\partial K}^2\\
&\le h_K^{-1}\|Q_0\vw-\vw\|_{\partial K}^2 \\
&\lesssim h_K^{-2} \|Q_0\vw-\vw\|_{K}^2 + \|\nabla(Q_0\vw-\vw)\|_K^2 \\
&\lesssim \|\nabla\vw\|_K^2.
\end{aligned}
$$
Similarly, one can show that for $K\in\T_h^D$,
$$
\begin{aligned}
h_K^{-1} \|(Q_0\vw-Q_b\vw)\cdot\vn\|_{\partial K}^2 &\lesssim
h_K^{-1}\|Q_0\vw-Q_b\vw\|_{\partial K}^2 \\ &\lesssim
h_K^{-1}\|Q_0\vw-\vw\|_{\partial K}^2
 \lesssim \|\nabla\vw\|_K^2.
 \end{aligned}
$$
By using Lemma \ref{lem:projections}, we have
$$
\|\nabla_w\cdot (Q_h\vw)\|_{\Omega}  = \|\pi_h(\nabla\cdot\vw)\|_{\Omega} \le \|\nabla\cdot\vw\|_{\Omega} \le \|\nabla\vw\|_{\Omega}.
$$
Finally, using the trace inequality,
$$
\|\mu^{1/2}\bbK^{-1/4}(Q_b\vw)\cdot\hvt\|_{\Gamma_{SD}} \lesssim \|Q_b\vw\|_{\Gamma_{SD}}
\le \|\vw\|_{\Gamma_{SD}} \lesssim \|\vw\|_{1,\Omega}.
$$
Combining all the above, we have $\|\vv\|_{\vV_h} \lesssim
\|\vw\|_1$, which completes the proof of the lemma. $\Box$
\medskip

By using Lemma \ref{lem:projections}, we can also prove the following result:
\begin{lemma} \label{lem:eq1}
The solution $\vu$ and $p$ to problem (\ref{eq:eq11}) satisfies
$$
a_h(Q_h \vu, \vv) + b_h(\vv, \bbQ_h p) = (f, \vv_0) + s(Q_h\vu, \vv)
+ l_S(\vv) - l_D(\vv) - l_{div}(\vv) - l_I (\vv),
$$
for all $\vv\in \vV_h$, where the linear functionals $l_S(\cdot)$,
$l_D(\cdot)$, $l_{div}(\cdot)$, and $l_I(\cdot)$ are defined by
$$
\begin{aligned}
l_S(\vv) &= 2\nu \sum_{K\in\T_h^S}  <\vv_0-\vv_b, (D(\vu)-\Pi_h D(\vu))\vn>_{\partial K}, \\
l_D(\vv) &= (\bbK^{-1} (\vu-Q_0\vu), \vv_0)_{\Omega_D}, \\[2mm]
l_{div}(\vv) &= \sum_{K\in\T_h}  <(\vv_0 - \vv_b)\cdot\vn, p-\bbQ_h p>_{\partial K} , \\
l_I (\vv) &= <\mu \bbK^{-1/2} (\vu_S-Q_b\vu_s) \cdot \hvt, \vv_b\cdot\hvt>_{\Gamma_{SD}}.
\end{aligned}
$$
\end{lemma}
{\bf Proof} Testing problem (\ref{eq:eq11}) with $\vv = \{\vv_0,
\vv_b\}\in \vV_h$ and using integration by parts, one gets
\begin{equation} \label{eq:erroreq-1}
\begin{aligned}
&\quad\, (\vf, \vv_0) \\&= (-\nabla\cdot(2\nu D(\vu)-pI), \vv_0)_{\Omega_S} + (\bbK^{-1} \vu, \vv_0)_{\Omega_D} + (\nabla p, \vv_0)_{\Omega_D}  \\
&= \sum_{K\in\T_h^S} \bigg( 2\nu(D(\vu), D(\vv_0))_K - 2\nu <\vv_0, D(\vu)\vn>_{\partial K} \bigg)
         + (\bbK^{-1} \vu, \vv_0)_{\Omega_D}\\
&\qquad +  \sum_{K\in\T_h} \bigg( -(\nabla\cdot\vv_0, p)_K + <\vv_0\cdot\vn, p>_{\partial K}\bigg) \\
&= \sum_{K\in\T_h^S} \bigg( 2\nu(D(\vu), D(\vv_0))_K - 2\nu <\vv_0 - \vv_b, D(\vu)\vn>_{\partial K} \bigg)
         + (\bbK^{-1} \vu, \vv_0)_{\Omega_D}\\
&\qquad +  \sum_{K\in\T_h} \bigg( -(\nabla\cdot\vv_0, p)_K + <(\vv_0-\vv_b)\cdot\vn, p>_{\partial K}\bigg)\\
&\qquad + <\mu \bbK^{-1/2} \vu_S\cdot \hvt, \vv_b\cdot\hvt>_{\Gamma_{SD}},
\end{aligned}
\end{equation}
where in the last step we have used $\vv_b = \vzero$ on $\partial \Omega$, the interface condition (\ref{eq:interfaceCondition}),
and the continuity of $\bbT(\vu,p)\vn$ and $p$ across the edges in $\E_{0,h}^S$ and $\E_{0,h}^D$, respectively.
More specifically, that is
$$
\begin{aligned}
&\sum_{K\in\T_h^S} 2\nu <\vv_b,  D(\vu)\vn>_{\partial K} - \sum_{K\in\T_h} <\vv_b\cdot\vn, p>_{\partial K} \\
=& \sum_{K\in\T_h^S} <\vv_b, \bbT(\vu, p)\vn>_{\partial K} - \sum_{K\in\T_h^D} <\vv_b, p\vn>_{\partial K} \\
=& \sum_{e\in \E_{0,h}^S} <\vv_b, [\bbT(\vu, p)\vn]>_e + \sum_{e\in \E_{h}^{SD}} <\vv_b, \bbT(\vu, p)\hvn>_e \\
&\quad - \sum_{e\in \E_{0,h}^D} <\vv_b, [p\vn]>_e - \sum_{e\in \E_{h}^{SD}} <\vv_b, -p\hvn>_e \\
= & \sum_{e\in \E_{h}^{SD}} <\vv_b, \bbT(\vu, p)\hvn>_e + \sum_{e\in \E_{h}^{SD}} <\vv_b, p\hvn>_e \\
= & -\sum_{e\in \E_{h}^{SD}} <\vv_b, \,\mu \bbK^{-1/2} \,
(\vu_S\cdot \hvt)\hvt>_e,
\end{aligned}
$$
where $[\cdot]$ denotes the jump, which is a notation borrowed from
the discontinuous Galerkin literature.

Now let us compute $a_h(Q_h \vu, \vv) + b_h(\vv, \bbQ_h p)$, where
$\vu$, $p$ are the solutions to (\ref{eq:eq11}) and $\vv\in \vV_h$.
Since $D_w(Q_h\vu) = \Pi_h D(\vu)$ is symmetric and by using the
properties of $Q_h$, we have
\begin{equation} \label{eq:erroreq2}
\begin{aligned}
&\quad\, a_h(Q_h \vu, \vv) + b_h(\vv, \bbQ_h p) \\&= 2\nu
(D_w(Q_h\vu), D_w(\vv))_{\Omega_S} + (\bbK^{-1} Q_0\vu,
\vv_0)_{\Omega_D}
    + s(Q_h\vu, \vv) \\
    &\qquad + <\mu\bbK^{-1/2} (Q_b\vu_S)\cdot\hvt, \vv_b\cdot\hvt>_{\Gamma_{SD}} - (\nabla_w\cdot\vv, \bbQ_h p) \\
    &= 2\nu (\Pi_h D(\vu), \nabla_w(\vv))_{\Omega_S} + (\bbK^{-1} Q_0 \vu, \vv_0)_{\Omega_D}
    + s(Q_h\vu, \vv) \\
    &\qquad + <\mu\bbK^{-1/2} (Q_b\vu_S)\cdot\hvt, \vv_b\cdot\hvt>_{\Gamma_{SD}} - (\nabla_w\cdot\vv, \bbQ_h p).
\end{aligned}
\end{equation}
Note that by condition (\ref{eq:parameters}) and the properties of
$L^2$ projections,
\begin{equation} \label{eq:erroreq3}
\begin{aligned}
&\, 2\nu (\Pi_h D(\vu), \nabla_w(\vv))_{\Omega_S} \\
=&\, 2\nu \sum_{K\in\T_h^S} \bigg(
   -(\vv_0, \nabla\cdot(\Pi_h D(\vu)))_K  + <\vv_b, \Pi_h D(\vu)\vn>_{\partial K} \bigg) \\
   =&\, 2\nu\sum_{K\in\T_h^S} \bigg( (\nabla\vv_0, \Pi_h D(\vu))_K  - <\vv_0-\vv_b, \Pi_h D(\vu)\vn>_{\partial K} \bigg) \\
   =&\, 2\nu\sum_{K\in\T_h^S} \bigg( (D(\vu), D(\vv_0))_K  - <\vv_0-\vv_b, \Pi_h D(\vu)\vn>_{\partial K} \bigg),
\end{aligned}
\end{equation}
and
\begin{equation} \label{eq:erroreq4}
\begin{aligned}
-(\nabla_w\cdot\vv, \bbQ_h p) &= \sum_{K\in\T_h} \bigg(
      (\vv_0, \nabla(\bbQ_h p))_K - <\vv_b\cdot\vn, \bbQ_h p>_{\partial K}  \bigg) \\
      &= \sum_{K\in\T_h} \bigg( -(\nabla\cdot\vv_0, \bbQ_h p)_K + <(\vv_0 - \vv_b)\cdot\vn, \bbQ_h p>_{\partial K}  \bigg) \\
      &= \sum_{K\in\T_h} \bigg( -(\nabla\cdot\vv_0, p)_K + <(\vv_0 - \vv_b)\cdot\vn, \bbQ_h p>_{\partial K}  \bigg).
\end{aligned}
\end{equation}
Substituting (\ref{eq:erroreq3}) and (\ref{eq:erroreq4}) into (\ref{eq:erroreq2}), and then using (\ref{eq:erroreq-1}),
this completes the proof of the lemma.
$\Box$
\medskip

Finally, we have
\begin{lemma} \label{lem:eq2}
The solution $\vu$ to problem (\ref{eq:eq11}) satisfies
$$
b_h(Q_h \vu, q) = -(g, q),
$$
for all $q\in \Psi_h$.
\end{lemma}
{\bf Proof} By Lemma \ref{lem:projections} and inequality (\ref{eq:parameters}), we have
$$
b_h(Q_h \vu, q) = -(\nabla_w\cdot(Q_h\vu), q) = -(\pi_h(\nabla\cdot\vu), q) = -(\nabla\cdot\vu, q) = -(g, q).
$$
This completes the proof of the lemma.
$\Box$

%%%%%%%%%%%%%%%%%%%%%%%%%%%%%%%%%%%%%%%%%%%%%%%%%%%%%%%%%%%%%%%%%%
%%%%%%%%%%%%%%%%%%%%%%%%%%%%%%%%%%%%%%%%%%%%%%%%%%%%%%%%%%%%%%%%%%
%%%%%%%%%%%%%%%%%%%%%%%%%%%%%%%%%%%%%%%%%%%%%%%%%%%%%%%%%%%%%%%%%%
\section{Error analysis} \label{sec:error}
In this section we derive an error estimation of the weak Galerkin
approximation (\ref{eq:weakGalerkinformulation}). Let $\vu$, $p$ be
the solution to problem (\ref{eq:eq11}), and $\vu_h = \{\vu_0,
\vu_b\}$, $p_h$ be the solution to the weak Galerkin formulation
(\ref{eq:weakGalerkinformulation}). Define the error by
$$
\eu = Q_h \vu - \vu_h = \{Q_0\vu-\vu_0, Q_b\vu-\vu_b\},\qquad \ep = \bbQ p - p_h.
$$
Then by Equation (\ref{eq:weakGalerkinformulation}), Lemma \ref{lem:eq1} and \ref{lem:eq2}, we clearly have
\begin{equation} \label{eq:erroreq}
\begin{aligned}
a_h(\eu, \vv) & + b_h(\vv, \ep) &&\\
   &= s(Q_h\vu, \vv) + l_S(\vv)-l_D(\vv) - l_{div}(\vv) - l_I (\vv) \qquad &&\textrm{for all }\vv\in \vV_h, \\
b_h(\eu, q) &= 0 \qquad && \textrm{for all } q\in \Psi_h.
\end{aligned}
\end{equation}

We shall first derive an upper bound for the right-hand side of
(\ref{eq:erroreq}). To this end, we start from getting an upper
bound of $\left( \sum_{K\in\T_h^S} h_K^{-1}
\|\vv_0-\vv_b\|_{\partial K}^2  \right)^{1/2}$, for any $\vv\in
\vV_h$. From the definition of $Q_b$ and $\|\cdot\|_{\vV_h}$, it is
clear that

\begin{lemma} \label{lem:functional1}
If $\alpha_S = \beta$, then for all $\vv\in \vV_h$ we have
$$
\left( \sum_{K\in\T_h^S} h_K^{-1} \|\vv_0-\vv_b\|_{\partial K}^2
\right)^{1/2} \lesssim \|\vv\|_{\vV_h}.
$$
\end{lemma}

We would like to derive the same bound for $\alpha_S = \beta+1$,
which turns out to be much more complicated and requires a discrete
Korn's inequality. The proof of the following lemma is too long and
hence is postponed to Appendix \ref{sec:appendix}.

\begin{lemma} \label{lem:functional2}
When $\beta\ge 1$, $\alpha_S = \beta+1$ and assume that all $K\in
\mathcal{T}_h^S$ are affine homeomorphic to a fixed finite set of
reference polygons, then for all $\vv\in \vV_h$ we have
$$
\left( \sum_{K\in\T_h^S} h_K^{-1} \|\vv_0-\vv_b\|_{\partial K}^2
\right)^{1/2} \lesssim \|\vv\|_{\vV_h}.
$$
\end{lemma}

For the case $\beta = 0$ and $\alpha_S = 1$, the discrete Korn's
inequality fails and we do not know if the same result as in lemmas
(\ref{lem:functional1}) and (\ref{lem:functional2}) holds or not.

Using lemmas (\ref{lem:functional1}) and (\ref{lem:functional2}), we
have
\begin{lemma} \label{lem:functionals}
Let $\vu$, $p$ be the solution to problem (\ref{eq:eq11}) and $\vv\in \vV_h$, then we have
$$
\begin{aligned}
s(Q_h\vu, \vv) &\lesssim \bigg(h^{r_S} \|\vu\|_{r_S+1,\Omega_S} + h^{r_D}\|\vu\|_{r_D+1,\Omega_D} \bigg)\|\vv\|_{\vV_h}, \\
l_S(\vv) &\lesssim  h^{r_\beta+1} \|\vu\|_{r_\beta+2,\Omega_S} \|\vv\|_{\vV_h},\\[2mm]
l_D(\vv) &\lesssim h^{r_D+1}\|\vu\|_{r_D+1,\Omega_D} \|\vv\|_{\vV_h},\\
l_{div}(\vv) &\lesssim \bigg( h^{t_S+1} \|p\|_{t_S+1,\Omega_S}
   +  h^{t_D+1} \|p\|_{t_D+1, \Omega_S} \bigg) \|\vv\|_{\vV_h}, \\
l_I (\vv) &\lesssim h^{r_{\beta}+1} \|\vu\|_{r_\beta+1,\Gamma_{SD}} \|\vv\|_{\vV_h}.
\end{aligned}
$$
where $0\le r_S\le \alpha_S$, $0\le r_D\le \alpha_D$, $0\le r_\beta \le \beta$, $0\le t_S\le \gamma_S$ and $0\le t_D\le \gamma_D$.
\end{lemma}
{\bf Proof}
By the definition of $s(\cdot, \cdot)$, the property of $Q_b$, the Schwartz inequality, inequality (\ref{eq:boundary2interior}),
and the approximation property of $Q_0$, we have
$$
\begin{aligned}
s(Q_h\vu, \vv) &= \rho_S \sum_{K\in \T_h^S} h_K^{-1} <Q_0\vu-\vu,
Q_b\vv_0-\vv_b>_{\partial K} \\
  &\quad + \rho_D \sum_{K\in\T_h^D} h_K^{-1} <(Q_0\vu-\vu)\cdot\vn, (Q_b\vv_0-\vv_b)\cdot\vn>_{\partial K} \\
 &\lesssim  \rho_S \left(\sum_{K\in \T_h^S} h_K^{-1} \|Q_0\vu-\vu\|_{\partial K}^2 \right)^{1/2}
                   \left(\sum_{K\in \T_h^S} h_K^{-1} \|Q_b\vv_0-\vv_b\|_{\partial K}^2 \right)^{1/2} \\
 &\quad + \rho_D \left(\sum_{K\in \T_h^D} h_K^{-1} \|Q_0\vu-\vu\|_{\partial K}^2 \right)^{1/2}
                 \left(\sum_{K\in \T_h^D} h_K^{-1} \|(Q_b\vv_0-\vv_b)\cdot\vn\|_{\partial K}^2 \right)^{1/2} \\
 &\lesssim \bigg(h^{r_S} \|\vu\|_{r_S+1,\Omega_S} + h^{r_D}\|\vu\|_{r_D+1,\Omega_D} \bigg)\|\vv\|_{\vV_h},
\end{aligned}
$$
where $0\le r_S\le \alpha_S$ and $0\le r_D\le \alpha_D$.

Next, by using the property of $L^2$ projection and lemmas
\ref{lem:functional1}-\ref{lem:functional2}, we have
$$
\begin{aligned}
l_S(\vv) &= 2\nu \sum_{K\in\T_h^S}  <\vv_0-\vv_b, (D(\vu)-\Pi_h D(\vu))\vn>_{\partial K} \\
  &\lesssim\left( \sum_{K\in\T_h^S} h_K^{-1} \|\vv_0-\vv_b\|_{\partial K}^2  \right)^{1/2}
      \left( \sum_{K\in\T_h^S} h_K \|(D(\vu)-\Pi_h D(\vu))\|_{\partial K}^2  \right)^{1/2}, \\
  &\lesssim h^{r_\beta+1} \|\vu\|_{r_\beta+2,\Omega_S} \|\vv\|_{\vV_h},
\end{aligned}
$$
where $0\le r_\beta \le \beta$. The estimate for $l_D(\vv)$ follows
directly from the approximation property of $Q_0$. Similarly
$$
\begin{aligned}
l_{div}(\vv) &= \sum_{K\in\T_h}  <(\vv_0 - \vv_b)\cdot\vn, p-\bbQ_h p>_{\partial K} \\
&\lesssim
   \left( \sum_{K\in\T_h^S} h_K^{-1} \|\vv_0-\vv_b\|_{\partial K}^2  \right)^{1/2} h^{t_S+1} \|p\|_{t_S+1,
   \Omega_S} \\
   &\qquad+ \left( \sum_{K\in\T_h^D} h_K^{-1} \|(\vv_0-\vv_b)\cdot\vn\|_{\partial K}^2  \right)^{1/2} h^{t_D+1} \|p\|_{t_D+1, \Omega_S} \\
   &\lesssim \bigg( h^{t_S+1} \|p\|_{t_S+1, \Omega_S}
   +  h^{t_D+1} \|p\|_{t_D+1, \Omega_S} \bigg) \|\vv\|_{\vV_h},
\end{aligned}
$$
where $0\le t_S\le \gamma_S$ and $0\le t_D\le \gamma_D$.
%Then, using the triangle inequality, Equation (\ref{eq:boundary2interior}), the inverse inequality, and the property of $Q_b$, $Q_0$,
%we have
%$$
%\begin{aligned}
%\sum_{K\in\T_h^D} h_K \|(\vv_0-\vv_b)\cdot\vn\|_{\partial K}^2 &\lesssim \sum_{K\in\T_h^D} h_K
%      (\|(\vv_0-Q_b\vv_0)\cdot\vn\|_{\partial K}^2 + \|(Q_b\vv_0-\vv_b)\cdot\vn\|_{\partial K}^2) \\
%      &\lesssim \sum_{K\in\T_h^D} h_K \|\vv_0-Q_0\vv_0\|_{\partial K}^2 + \|\vv\|_{\vV_0}^2 \\
%      &\lesssim  \sum_{K\in\T_h^D} \|\vv_0 - Q_0 \vv_0\|_K^2 + \|\vv\|_{\vV_0}^2 \\
%      &\lesssim \sum_{K\in\T_h^D} \|\vv_0 \|_K^2 + \|\vv\|_{\vV_0}^2 \lesssim \|\vv\|_{\vV_0}^2.
%\end{aligned}
%$$

%One may wonder if the term $\left( \sum_{K\in\T_h^S} h_K^{-1} \|\vv_0-\vv_b\|_{\partial K}^2  \right)^{1/2}$
%can be treated similarly. Unfortunately it is not as easy as treating
%$\left( \sum_{K\in\T_h^D} h_K \|(\vv_0-\vv_b)\cdot\vn\|_{\partial K}^2  \right)^{1/2}$,
%as will be discussed after we complete the proof of this lemma.

Finally,
$$
\begin{aligned}
l_I (\vv) &= <\mu \bbK^{-1/2} (\vu_S-Q_b\vu_s) \cdot \hvt, \vv_b\cdot\hvt>_{\Gamma_{SD}} \\
&\lesssim \|\vu_s - Q_b \vu_s\|_{\Gamma_{SD}} \|\vv\|_{\vV_h} \\
&\lesssim h^{r_{\beta}+1} \|\vu\|_{r_\beta+1,\Gamma_{SD}} \|\vv\|_{\vV_h}.
\end{aligned}
$$
This completes the proof of the lemma. $\Box$
\medskip

Now we are able to write the error estimate:
\begin{theorem}
Let $\alpha_s$ and $\beta$ satisfy the conditions in Lemma
\ref{lem:functional1} and \ref{lem:functional2}, then the error
$\eu$ and $\ep$ satisfies
$$
\begin{aligned}
\|\eu\|_{\vV_h}+\|\ep\|&\lesssim h^{r_S} \|\vu\|_{r_S+1,\Omega_S} +
h^{r_D}\|\vu\|_{r_D+1,\Omega_D} \\
&\qquad  + h^{r_\beta+1} (\|\vu\|_{r_\beta+2,\Omega_S}+ \|\vu\|_{r_\beta+1,\Gamma_{SD}}) \\
&\qquad  + h^{t_S+1} \|p\|_{t_S+1, \Omega_S}    +  h^{t_D+1} \|p\|_{t_D+1, \Omega_S},
  \end{aligned}
$$
where $0\le r_S\le \alpha_S$, $0\le r_D\le \alpha_D$, $0\le r_\beta \le \beta$, $0\le t_S\le \gamma_S$ and $0\le t_D\le \gamma_D$.
\end{theorem}
{\bf Proof}
By setting $\vv = \eu$ and $q = \ep$ in (\ref{eq:erroreq}) and then subtract these two equations, we have
$$
\|\eu\|_{\vV_h}^2 = a_h(\eu,\eu) = s(Q_h\vu, \eu) + l_S(\eu) -
l_D(\eu) - l_{div}(\eu) - l_I (\eu).
$$
Applying Lemma \ref{lem:functionals}, this completes the proof for
$\|\eu\|_{\vV_h}$.

To estimate $\|\ep\|$, note that from (\ref{eq:erroreq}), we have
$$
\begin{aligned}
&b_h(\vv, \ep) \\
= &s(Q_h\vu, \vv) + l_S(\vv) - l_D(\vv) - l_{div}(\vv) - l_I (\vv) - a_h(\eu, \vv) \\
\lesssim & \bigg(h^{r_S} \|\vu\|_{r_S+1,\Omega_S} + h^{r_D}\|\vu\|_{r_D+1,\Omega_D}
  + h^{r_\beta+1} (\|\vu\|_{r_\beta+2,\Omega_S}+\|\vu\|_{r_\beta+1,\Gamma_{SD}}) \\
  &\quad
  + h^{t_S+1} \|p\|_{t_S+1, \Omega_S}    +  h^{t_D+1} \|p\|_{t_D+1, \Omega_S} + \|\eu\|_{\vV_h} \bigg) \|\vv\|_{\vV_h}.
\end{aligned}
$$
Then by the discrete inf-sup condition stated in Lemma \ref{lem:inf-sup}, this completes the proof for $\|\ep\|$.
$\Box$

\begin{remark}
Using the condition (\ref{eq:parameters}), one can see that,
assuming the solution to (\ref{eq:eq11}) be as smooth as possible,
we expect to have
$$
\|\eu\|_{\vV_h}+\|\ep\| \lesssim h^{\beta} + h^{\gamma_S+1} + h^{\gamma_D+1}.
$$
Therefore, the best choice seems to be setting $\alpha_S=\alpha_D = \beta = j $ and $\gamma_S =\gamma_D = j-1$,
for $j\ge 1$.
\end{remark}
%%%%%%%%%%%%%%%%%%%%%%%%%%%%%%%%%%%%%%%%%%%%%%%%%%%%%%%%%%%%%%%%%%
%%%%%%%%%%%%%%%%%%%%%%%%%%%%%%%%%%%%%%%%%%%%%%%%%%%%%%%%%%%%%%%%%%
%%%%%%%%%%%%%%%%%%%%%%%%%%%%%%%%%%%%%%%%%%%%%%%%%%%%%%%%%%%%%%%%%%
\section{Numerical results} \label{sec:numerical}

In this section we report some numerical results from solving the
following test problem. Let $\Omega_S = (0,\pi)\times(0,1)$,
$\Omega_D = (0,\pi)\times(-1,0)$ and the interface be $\Gamma_{SD} =
\{0 < x < \pi,\, y=0\}$. Set the coefficients to be
$$
\nu = 1,\qquad \bbK = I, \qquad \mu = 1.
$$
Under the given domain and coefficients, the following set of
functions is known \cite{Chen10} to satisfy the coupled Darcy-Stokes
equation (\ref{eq:eq11}) with the interface conditions
(\ref{eq:interface1})-(\ref{eq:interface3}):
$$
\begin{aligned}
\vu_S &= \begin{bmatrix}v'(y)\cos x\\ v(y)\sin x\end{bmatrix}
\quad\textrm{where}\qquad v(y) = \frac{1}{\pi^2} \sin^2 (\pi y) - 2,
\\[2mm]
p_S &= \sin x\,\sin y,\\[2mm]
p_D &= (e^y-e^{-y})\sin x \qquad\textrm{and}\qquad \vu_D = -\nabla
p_D.
\end{aligned}
$$
Use this set of functions, one can compute the force functions $\vf$
and $g$, as well as the boundary conditions $\vu_S|_{\Gamma_S}$ and
$\vu_D\cdot\vn|_{\Gamma_D}$. This gives a complete set of data for
the coupled Darcy-Stokes problem
(\ref{eq:eq11})-(\ref{eq:interface3}), with the exact solution
known. Of course, we need to subtract $p$ by $\int_\Omega p\, dx$ in
order to make it mean value free.

We test the weak Galerkin approximation
(\ref{eq:weakGalerkinformulation}) for this test problem with the
following settings. Both the Stokes and the Darcy side are divided
into $n\times n$ grids, which combined together forms a $(2n)\times
n$ rectangular mesh. The weak Galerkin space is chosen so that
$$
\alpha_S = \alpha_D = \beta = 1,\qquad \gamma_S = \gamma_D = 0.
$$
According to the analysis given in this paper, the discrete system
(\ref{eq:weakGalerkinformulation}) is well-posed and we expect it to
provide an approximation error of at least
$$
\|\eu\|_{\vV_h}+\|\ep\| = O(h).
$$

The error terms in the theoretical analysis are defined as $\eu =
Q_h \vu - \vu_h$ and $\ep = \bbQ p - p_h$. In practice, for
simplicity, we made some modification. Define by $I_h \vu$ an
interpolation of $\vu$ in $\vV_h$ such that: its value on an edge is
the usual $P_1$ nodal value interpolation using two end points of
the edge; and its value on a rectangle is the $P_1$ nodal value
interpolation using three vertices of the rectangle, the lower-left
corner, the lower-right corner, and the upper-left corner. Define by
$J_h p$ an interpolation of $p$ in $\Psi_h$ such that its value in
each rectangle is a constant equal to the value of $p$ at the center
of this rectangle. Then, we consider the following modified error
terms
$$
\heu = I_h\vu-\vu_h,\qquad \hep = J_h p - p_h.
$$
By the approximation property of projections and interpolations, we
expect that $\|\heu\|_{\vV_h}+\|\hep\|$ has the same order as
$\|\eu\|_{\vV_h}+\|\ep\|$.

We computed the following norms and seminorms of the error: on the
Stokes side are $\|\nabla_w \heu\|_{\Omega_S}$,
$\|(\heu)_0\|_{\Omega_S}$, and $\|\hep\|_{\Omega_S}$; on the Darcy
side are $\|(\heu)_0\|_{\Omega_D}$ and $\|\hep\|_{\Omega_D}$.
According to the theoretical error estimate, we expect at least
$$
\begin{aligned}
\|\nabla_w \heu\|_{\Omega_S} &= O(h),\qquad &\|\hep\|_{\Omega_S}
&=O(h), \\
\|(\heu)_0\|_{\Omega_D} &= O(h),\qquad &\|\hep\|_{\Omega_D}&= O(h).
\end{aligned}
$$
We did not have theoretical error estimate for
$\|(\heu)_0\|_{\Omega_S}$, although by experience from using the
duality argument, we expect that the optimal case for this term is
$$
\|(\heu)_0\|_{\Omega_S} = O(h^2).
$$

In the numerical experiment, we picked the stabilization constants
$\rho_S = \rho_D = \rho$ to be $0.01$, $1$ and $100$. Numerical
results are reported in the tables
\ref{table:rho001}-\ref{table:rho100} and Figure \ref{fig:rho}.

\begin{table}
\begin{center}
\caption{Error for $\rho= 0.01$. Order $O(h^r)$ computed from $n=16$
to $n=128$.} \label{table:rho001}
\begin{tabular}{|c|c|c|c|c|c|}
\hline $n$ & $\|\nabla_w \heu\|_{\Omega_S}$ &
$\|(\heu)_0\|_{\Omega_S}$ & $\|\hep\|_{\Omega_S}$ &
$\|(\heu)_0\|_{\Omega_D}$ & $\|\hep\|_{\Omega_D}$ \\ \hline 8 &
0.76224 &  2.26639 & 0.84301 & 2.54274 & 0.91539 \\ \hline 16 &
0.30306 &  0.26226 & 0.25407 & 1.56049 & 0.56486 \\ \hline 32 &
0.14960 &  0.03332 & 0.08316 & 0.65226 & 0.24125 \\ \hline 64 &
0.07461 &  0.00486 & 0.02365 & 0.20346 & 0.07536 \\ \hline 128 &
0.03719 &  0.00089 & 0.00615 & 0.05691 & 0.02016 \\ \hline $O(h^r)$,
$r=$ & 1.0083 & 2.7383  &  1.7917  &  1.6012  &  1.6103 \\ \hline
\end{tabular}
\end{center}
\end{table}

\begin{table}
\begin{center}
\caption{Error for $\rho = 1$.  Order $O(h^r)$ computed from $n=8$
to $n=128$.} \label{table:rho1}
\begin{tabular}{|c|c|c|c|c|c|}
\hline $n$ & $\|\nabla_w \heu\|_{\Omega_S}$ &
$\|(\heu)_0\|_{\Omega_S}$ & $\|\hep\|_{\Omega_S}$ &
$\|(\heu)_0\|_{\Omega_D}$ & $\|\hep\|_{\Omega_D}$ \\ \hline 8 &
0.56159 &  0.03842 & 0.07539 & 0.18953 & 0.07511 \\
\hline 16 & 0.28729 & 0.00850 & 0.02055 & 0.06858 & 0.01953  \\
\hline 32 & 0.14443 & 0.00204 & 0.00538 & 0.02925 & 0.00492 \\
\hline 64 & 0.07231 & 0.00050 & 0.00137 & 0.01381 & 0.00123 \\
\hline 128& 0.03616 & 0.00012 & 0.00035 & 0.00678 & 0.00031 \\
\hline $O(h^r)$, $r=$ &  0.9904 &   2.0622  &  1.9402 &   1.1924 &
1.9846 \\ \hline
\end{tabular}
\end{center}
\end{table}

\begin{table}
\begin{center}
\caption{Error for $\rho = 100$.  Order $O(h^r)$ computed from $n=8$
to $n=128$.} \label{table:rho100}
\begin{tabular}{|c|c|c|c|c|c|}
\hline $n$ & $\|\nabla_w \heu\|_{\Omega_S}$ &
$\|(\heu)_0\|_{\Omega_S}$ & $\|\hep\|_{\Omega_S}$ &
$\|(\heu)_0\|_{\Omega_D}$ & $\|\hep\|_{\Omega_D}$ \\ \hline 8 &
0.47789 & 0.03379 & 0.31537 & 0.06226 & 0.16416 \\ \hline 16 &
0.24268 & 0.01117 & 0.09409 & 0.02190 & 0.04211 \\ \hline 32 &
0.12079 & 0.00325 & 0.02371 & 0.00950 & 0.01061 \\ \hline 64 &
0.06017 & 0.00086 & 0.00583 & 0.00457 & 0.00264 \\ \hline 128 &
0.03004 & 0.00022 & 0.00144 & 0.00227 & 0.00066 \\ \hline $O(h^r)$,
$r=$ & 0.9995  &  1.8266  &  1.9552  &  1.1817  &  1.9921 \\ \hline
\end{tabular}
\end{center}
\end{table}

In Figure \ref{fig:rho}, we conveniently denote the norms
$\|\nabla_w \heu\|_{\Omega_S}$, $\|(\heu)_0\|_{\Omega_S}$,
$\|\hep\|_{\Omega_S}$, $\|(\heu)_0\|_{\Omega_D}$,
$\|\hep\|_{\Omega_D}$  by H1 $\vu_S$, L2 $\vu_S$, L2 $p_S$, L2
$\vu_D$, and L2 $p_D$. For $\rho=1$ and $\rho=100$, the asymptotic
behavior of the errors are clearly seen from Figure \ref{fig:rho},
as the curves are almost straight lines. For $\rho=0.01$, it seems
that the asymptotic behavior deteriorates when $h$ is large. But as
$h$ becomes smaller, the convergence rates get better. However, we
notice that for all three values of $\rho$, the orders of
$\|(\heu)_0\|_{\Omega_S}$ are approximately equal to or higher than
$O(h^2)$, while the order of other errors are approximately equal to
or higher than $O(h)$, which are guaranteed by the theoretical
analysis. One important and interesting question is how to pick the
best parameter $\rho$. One may use the techniques proposed in
\cite{LazarovLu}. It is beyond the scope of this paper, but suitable
for a future research topic.

From the results of $\rho=1$ and $\rho=100$, it seems that the both
$\|\hep\|_{\Omega_S}$ and $\|\hep\|_{\Omega_D}$ also achieves
$O(h^2)$ convergence. This might be caused by super-convergence on
uniform rectangular meshes. We will not explore the
super-convergence effect here. The super-convergence of a primal
based formulation for the Darcy-Stokes equation has been discussed
in \cite{Chen10}.

\begin{figure}
\begin{center}
\includegraphics[width=6cm]{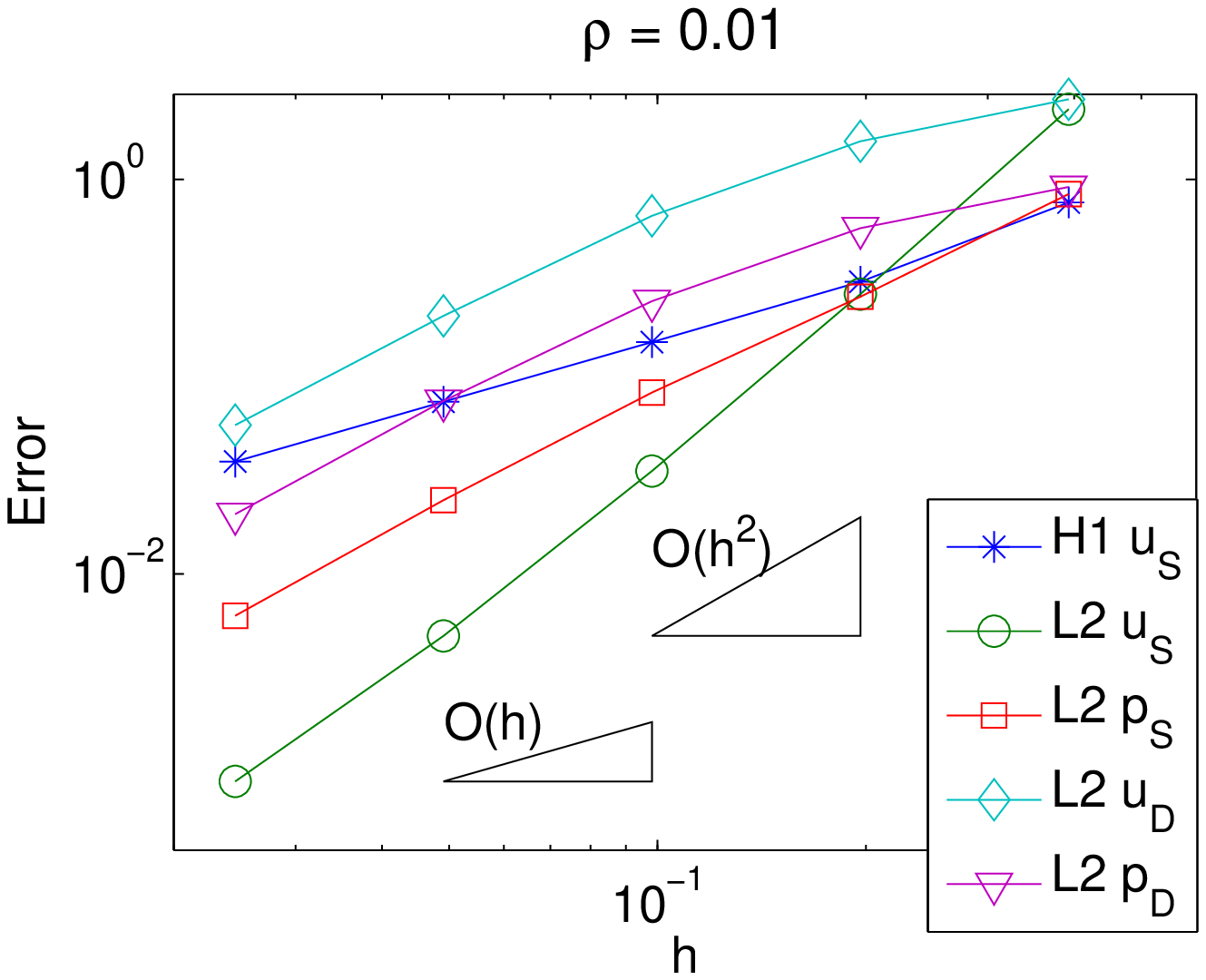}\;\includegraphics[width=6cm]{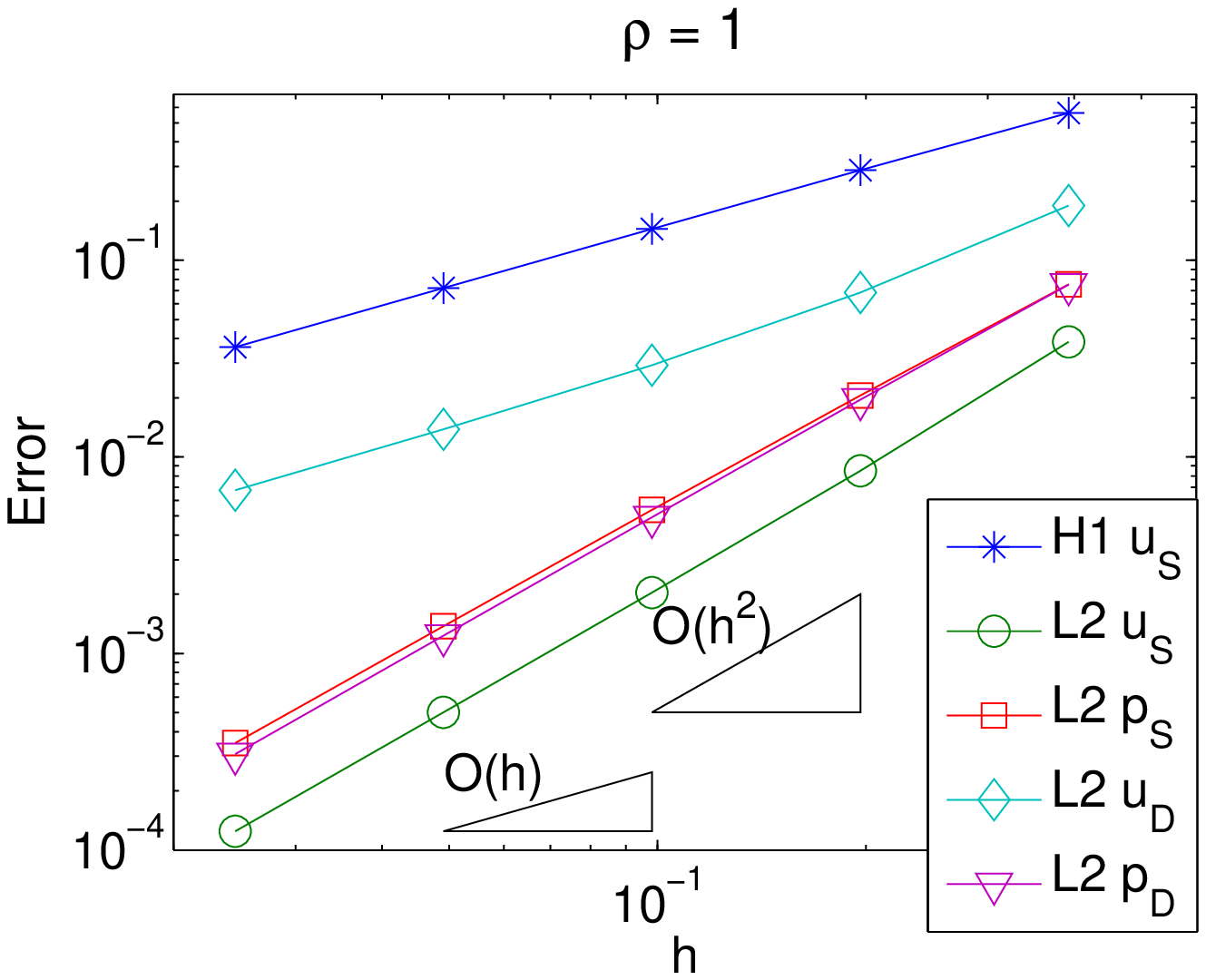}\\
\includegraphics[width=6cm]{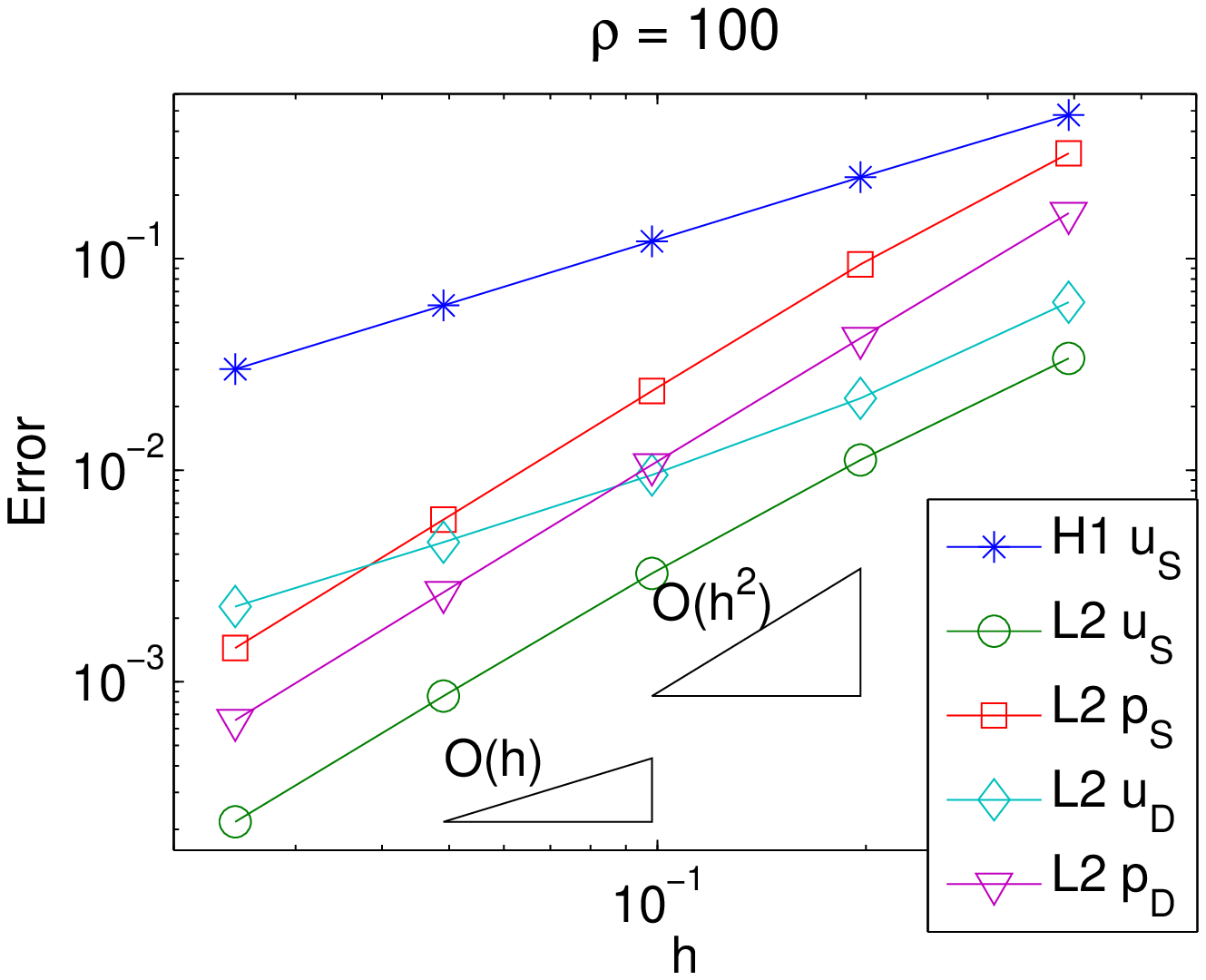}
\caption{Convergence rates, $\rho = 0.01$, $1$ and $100$.}
\label{fig:rho}
\end{center}
\end{figure}

%%%%%%%%%%%%%%%%%%%%%%%%%%%%%%%%%%%%%%%%%%%%%%%%%%%%%%%%%%%%%%%%%%
%%%%%%%%%%%%%%%%%%%%%%%%%%%%%%%%%%%%%%%%%%%%%%%%%%%%%%%%%%%%%%%%%%
%%%%%%%%%%%%%%%%%%%%%%%%%%%%%%%%%%%%%%%%%%%%%%%%%%%%%%%%%%%%%%%%%%
\appendix
\section{Proof of Lemma \ref{lem:functional2}} \label{sec:appendix}

 By using the triangle
inequality, Equation (\ref{eq:boundary2interior}), the inverse
inequality, the property of $Q_b$ and $Q_0$, for all $\vv\in \vV_h$
we have
$$
\begin{aligned}
\sum_{K\in\T_h^S} h_K^{-1} \|\vv_0-\vv_b\|_{\partial K}^2 &\lesssim
\sum_{K\in\T_h^S} h_K^{-1}
  (\|\vv_0-Q_b\vv_0\|_{\partial K}^2 + \|Q_b\vv_0-\vv_b\|_{\partial K}^2) \\
  & \lesssim \sum_{K\in\T_h^S} h_K^{-1} \|\vv_0-Q_0\vv_0\|_{\partial K}^2 + \|\vv\|_{\vV_h}^2 \\
  &\lesssim \sum_{K\in\T_h^S} h_K^{-2} \|\vv_0-Q_0\vv_0\|_K^2 + \|\vv\|_{\vV_h}^2 \\
  &\lesssim \sum_{K\in\T_h^S} \|\nabla \vv_0\|_K^2 +
  \|\vv\|_{\vV_h}^2.
  %&\lesssim \|\nabla_w \vv\|_{\Omega_S}^2 + \|\vv\|_{\vV_h}^2
\end{aligned}
$$
Now the difficulty is to bound $\sum_{K\in\T_h^S} \|\nabla
\vv_0\|_K^2$ by $\|\vv\|_{\vV_h}^2$. By Lemma A.2 in
\cite{wy-Stokes}, one has $\|\nabla \vv_0\|_K^2 \lesssim \|\nabla_w
\vv\|_K^2 + h_K^{-1} \|Q_b\vv_0 - \vv_b\|_{\partial K}^2$ and this
seems to be a possible solution. However, on second thought,
$\|\nabla_w \vv\|_{\Omega_S}$ is not necessarily bounded by
$\|\vv\|_{\vV_h}$, which indeed contains $\|D_w(\vv)\|_{\Omega_S}$.
Here one obviously needs a discrete Korn's inequality involving the
weak gradient.

It turns out to be easier to first apply a discrete Korn's
inequality to $\nabla \vv_0$ instead of trying to bound it using
$\nabla_w \vv$. By \cite{Brenner03}, when all $K\in \mathcal{T}_h^S$
are affine homeomorphic to a fixed finite set of reference polygons,
one has
\begin{equation} \label{eq:discreteKorn}
\begin{aligned}
\sum_{K\in\T_h^S} \|\nabla \vv_0\|_K^2 &\lesssim
  \sum_{K\in\T_h^S} \|D(\vv_0)\|_K^2
  + \sup_{\stackrel{\mathbf{m}\in RM,\,\|\mathbf{m}\|_{\Gamma_S} = 1}{\int_{\Gamma_S}\mathbf{m}\, ds =
  \mathbf{0}}} \left(\int_{\Gamma_S} \vv_0\cdot\mathbf{m}\, ds
  \right)^2 \\[2mm]
  &\qquad + \sum_{e\in\E_{0,h}^S} \|\pi_e [\vv_0]\|_e^2,
  \end{aligned}
\end{equation}
where $RM$ is the space of rigid body motions, $\pi_e$ is the $L^2$
projection onto $(P_1(e))^2$ and $[\cdot]$ denotes the jump on
edges. Next, we estimate the right-hand side of
(\ref{eq:discreteKorn}) one-by-one.

Similar to Lemma A.2 in \cite{wy-Stokes}, we have
\begin{lemma} \label{lem:vv0}
For $\vv\in \vV_h$ and any $K\in \T_h^S$, we have
$$
\|D(\vv_0)\|_K \lesssim \|D_w(\vv)\|_K + h_K^{-1/2} \|Q_b\vv_0 -
\vv_b\|_{\partial K}.
$$
\end{lemma}
\begin{proof} Note that
$$
\begin{aligned}
(D(\vv_0),\, D(\vv_0))_K &= (D(\vv_0),\, \nabla \vv_0)_K \\
&= -(\vv_0,\, \nabla\cdot D(\vv_0))_K +
<\vv_0,\,D(\vv_0)\cdot\vn>_{\partial K} \\
&= (\nabla_w \vv,\,D(\vv_0))_K + <\vv_0 - \vv_b,\,D(\vv_0)\cdot\vn>_{\partial K} \\
&= (D_w(\vv),\,D(\vv_0))_K + <Q_b \vv_0 -
\vv_b,\,D(\vv_0)\cdot\vn>_{\partial K}.
\end{aligned}
$$
The lemma then follows from the Schwarz inequality, Inequality
(\ref{eq:boundary2interior}) and the inverse inequality.
\end{proof}

The estimate of the second and the third term in the right-hand side
of (\ref{eq:discreteKorn}) requires $\beta\ge 1$. Indeed, when
$\beta\ge 1$, since $RM\subset (P_\beta)^2$ and $\vv_b$ vanishes on
$\Gamma_S$, we have
\begin{equation} \label{eq:Korn1}
\begin{aligned}
\sup_{\stackrel{\mathbf{m}\in RM,\,\|\mathbf{m}\|_{\Gamma_S} =
1}{\int_{\Gamma_S}\mathbf{m}\, ds =
  \mathbf{0}}} \left(\int_{\Gamma_S} \vv_0\cdot\mathbf{m}\, ds \right)^2
&= \sup_{\stackrel{\mathbf{m}\in RM,\,\|\mathbf{m}\|_{\Gamma_S} =
1}{\int_{\Gamma_S}\mathbf{m}\, ds =
  \mathbf{0}}} \left(\int_{\Gamma_S} (Q_b\vv_0-\vv_b)\cdot\mathbf{m}\, ds \right)^2
  \\
  &\lesssim \|Q_b\vv_0-\vv_b\|_{\Gamma_{S}}^2,
\end{aligned}
\end{equation}
and since $\pi_e[\vv_0]\in (P_1(e))^2 \subset (P_{\beta}(e))^2$, we
have
\begin{equation} \label{eq:Korn2}
\begin{aligned}
\sum_{e\in\E_{0,h}^S} \|\pi_e [\vv_0]\|_e^2 &\lesssim
\sum_{e\in\E_{0,h}^S} \|Q_b [\vv_0]\|_e^2 \\
&\lesssim \sum_{K\in\T_h^S} \|Q_b \vv_0 - \vv_b\|_{\partial K}^2.
\end{aligned}
\end{equation}

Combining the above analysis, by using inequalities
(\ref{eq:discreteKorn}), (\ref{eq:Korn1}), (\ref{eq:Korn2}), Lemma
\ref{lem:vv0} and the fact that $O(1)\le h_K^{-1/2}$, this completes
the proof of Lemma \ref{lem:functional2}.

\bigskip
\textbf{ACKNOWLEDGMENTS:} Chen is supported by the Key Project
National Science Foundation of China(91130004) and Natural Science
Foundation of China (11171077, 11331004).  Wang thanks the Key
Laboratory of Mathematics for Nonlinear Sciences (EZH1411108/001) of
Fudan University, and the Ministry of Education of China \& State
Administration of Foreign Experts Affairs of China under the 111
project grant (B08018), for the support during her visit.
 %   The authors are also grateful to the anonymous referees for their  helpful comments
%and suggestions which greatly improved the quality of this paper.

%%%%%%%%%%%%%%%%%%%%%%%%%%%%%%%%%%%%%%%%%%%%%%%%%%%%%%%%%%%%%%%%%%
%%%%%%%%%%%%%%%%%%%%%%%%%%%%%%%%%%%%%%%%%%%%%%%%%%%%%%%%%%%%%%%%%%
%%%%%%%%%%%%%%%%%%%%%%%%%%%%%%%%%%%%%%%%%%%%%%%%%%%%%%%%%%%%%%%%%%

\end{document}